\pdfoutput=1
\def\arXiv{1} 
\documentclass[11pt]{article} 

\newcommand{\notarxiv}[1]{foo}
\newcommand{\arxiv}[1]{ba}
\ifdefined \arXiv
\renewcommand{\arxiv}[1]{#1}%
\renewcommand{\notarxiv}[1]{\ignorespaces}%
\else%
\renewcommand{\arxiv}[1]{\ignorespaces}%
\renewcommand{\notarxiv}[1]{#1}%
\fi%

\usepackage{thmtools}
\usepackage{thm-restate}

 \notarxiv{
 	\undef\corollary
	 \undef\definition
	 \undef\assumption

	\setcitestyle{numbers,square,comma} 
	\declaretheorem[name=Lemma,sibling=theorem]{lem}
	\declaretheorem[name=Proposition,sibling=theorem]{prop}
	\declaretheorem[name=Corollary,sibling=theorem]{corollary}
	\declaretheorem[name=Assumption,sibling=theorem]{assumption}
	\declaretheorem[name=Definition,sibling=theorem]{definition}

	\undef\lemma
	\undef\proposition	

}

\arxiv{
	\usepackage{amsthm}
	\usepackage[linesnumbered,ruled,vlined,boxed,algo2e]{algorithm2e}

	\usepackage[dvipsnames]{xcolor}
	\usepackage[top=1in, right=1in, left=1in, bottom=1in]{geometry}
	\usepackage[numbers,square]{natbib}
	\usepackage{url}
	\usepackage[hidelinks]{hyperref}
	\hypersetup{
		colorlinks=true,
		linkcolor=blue!70!black,
		citecolor=blue!70!black,
		urlcolor=blue!70!black}
	
	\theoremstyle{plain}
	
	\newtheorem{theorem}{Theorem}
	\newtheorem{lem}{Lemma}
	
	\newtheorem{assumption}{Assumption}
	\newtheorem{prop}{Propostion}
	\newtheorem{corollary}{Corollary}
	\newtheorem{definition}{Definition}
	
	\theoremstyle{definition}

	\newtheorem*{example*}{Example}

}

\usepackage{enumitem}
\usepackage{amssymb}
\usepackage{amsbsy}
\usepackage{amsmath}

\usepackage{cleveref}
\Crefname{assumption}{Assumption}{Assumptions}
\Crefname{lem}{Lemma}{Lemmas}
\Crefname{prop}{Proposition}{Propositions}

\usepackage{amsfonts}
\usepackage{latexsym}
\usepackage{graphicx}
\usepackage{color}
\usepackage{xifthen}
\usepackage{xspace}
\usepackage{mathtools}
\usepackage{bbm}
\usepackage{multirow}
\usepackage[normalem]{ulem}
\usepackage{array}
\usepackage[utf8]{inputenc}
\usepackage[T1]{fontenc}
\usepackage{etoolbox}
\newtoggle{restatements}

\newtoggle{heavyplots}

\usepackage{xfrac}
\usepackage{xargs}

\usepackage{titlesec}

\usepackage{setspace}

\usepackage{esvect}

\usepackage{nicefrac}

\usepackage{cases}
\usepackage{empheq}

 \usepackage{booktabs}
 \usepackage{multirow}
 
 \usepackage{caption}

\usepackage{amsmath}

\DeclarePairedDelimiter{\abs}{\lvert}{\rvert} %
\DeclarePairedDelimiter{\brk}{[}{]}
\DeclarePairedDelimiter{\crl}{\{}{\}}
\DeclarePairedDelimiter{\prn}{(}{)}

\DeclarePairedDelimiter{\norm}{\|}{\|}

\DeclarePairedDelimiter{\ceil}{\lceil}{\rceil}
\DeclarePairedDelimiter{\floor}{\lfloor}{\rfloor}

\newcommand{\overeq}[1]{\overset{#1}{=}}
\newcommand{\overle}[1]{\overset{#1}{\le}}

\NewDocumentCommand\Ex{s O{} m }{%
	\mathbb{E}%
	\begingroup
	\IfBooleanTF{#1}
	{\ExInn*{#3}}
	{\ExInn[#2]{#3}}%
	\endgroup
}

\DeclarePairedDelimiterX\ExInn[1]{[}{]}{%
	\activatebar
	#1%
}

\RenewDocumentCommand\Pr{sO{}r()}{%
	\mathbb{P}%
	\begingroup
	\IfBooleanTF{#1}
	{\PrInn*{#3}}
	{\PrInn[#2]{#3}}%
	\endgroup
}

\DeclarePairedDelimiterX\PrInn[1](){%
	\activatebar
	#1%
}

\newcommand{\activatebar}{%
	\begingroup\lccode`~=`|
	\lowercase{\endgroup\def~}{\;\delimsize\vert\;}%
	\mathcode`|=\string"8000
}

\newcommand\numberthis{\addtocounter{equation}{1}\tag{\theequation}}

\newcommand{\sm}{S} %

\newcommand{\R}{\mathbb{R}} %
\newcommand{\N}{\mathbb{N}} %
\newcommand{\E}{\mathbb{E}} %
\renewcommand{\P}{\mathbb{P}}	%
\usepackage{xcolor}

\SetCommentSty{mycommfont}
\SetKwInput{KwInput}{Input} 
\SetKwInput{KwParameter}{Parameters}                %
\SetKwInput{KwOutput}{Output}              %
\SetKwInput{KwReturn}{Return}              %
\SetKwComment{Comment}{$\triangleright$\ }{}
\SetKwInOut{KwParameters}{Parameters}
\SetCommentSty{mycommfont}

\providecommand{\abs}{\mathop{\rm abs}}

\newcommand{\half}{\frac{1}{2}}

\newcommand{\defeq}{\coloneqq}

\newcommand{\grad}{\nabla}

\newcommand{\xset}{\mathcal{X}}

\newcommand{\inner}[2]{\left<#1,#2\right>}

\newcommand{\veps}{\varepsilon}

\newcommand{\etalo}{\eta_{\mathrm{lo}}}
\newcommand{\etainit}{\eta_\varepsilon}
\newcommand{\etahi}{\eta_{\mathrm{hi}}}
\newcommand{\etamid}{\eta_{\mathrm{mid}}}
\newcommand{\etaout}{\eta_{\mathrm{o}}}
\newcommand{\etamax}{\eta_{\max}}
\newcommand{\etalofinal}{\etalo^\star}
\newcommand{\etahifinal}{\etahi^\star}
\newcommand{\reps}{r_\varepsilon}

\newcommand{\gdiff}[1][i]{\Delta_{#1}}
\newcommand{\G}[1][T]{G_{#1}}
\newcommand{\rbar}[1][T]{\bar{r}_{#1}}
\newcommand{\dbar}[1][T]{\bar{d}_{#1}}
\renewcommand{\d}[1][0]{d_{#1}}
\renewcommand{\r}[1][0]{r_{#1}}

\newcommand{\Gcoef}{\alpha}
\newcommand{\Ccoef}{\beta}
\newcommand{\GcoefK}[1][k]{\alpha^{(#1)}}
\newcommand{\CcoefK}[1][k]{\beta^{(#1)}}

\newcommand{\ev}[1][T,\Gcoef,\Ccoef]{\mathfrak{E}_{#1}}
\newcommand{\proj}{\Pi_1}
\newcommand{\filt}{\mathcal{F}}
\NewDocumentCommand{\llft}{ O{\delta} O{t} }{\lambda_{#2}(#1)}

\newcommand{\bt}{\phi}

\newcommand{\rootfinding}{\textsc{RootFindingBisection}}

\newcommand{\opt}{_\star}
\newcommand{\xopt}{x\opt}

\notarxiv{
\usepackage{times}
\usepackage[utf8]{inputenc} %
}

\notarxiv{
\title[Making SGD Parameter-Free]{Making SGD Parameter-Free}

\coltauthor{%
 \Name{Yair Carmon} \Email{ycarmon@tauex.tau.ac.il}\\
 \addr Tel Aviv University
 \AND
 \Name{Oliver Hinder} \Email{ohinder@pitt.edu}\\
 \addr University of Pittsburgh%
}
}

\arxiv{
	\title{Making SGD Parameter-Free}
	\author{Yair Carmon\\
		\href{mailto:ycarmon@tauex.tau.ac.il}{\texttt{ycarmon@cs.tau.ac.il}}\and Oliver Hinder \\
		\href{ohinder@pitt.edu}{\texttt{ohinder@pitt.edu}}
}
\date{}
}

\begin{document}

\maketitle

\begin{abstract}%
We develop an algorithm for parameter-free stochastic convex optimization (SCO) whose rate of convergence is only a double-logarithmic factor larger than the optimal rate for the corresponding known-parameter setting. In contrast, the best previously known rates for parameter-free SCO are based on online parameter-free regret bounds, which contain unavoidable excess logarithmic terms compared to their known-parameter counterparts. Our algorithm is conceptually simple, has high-probability guarantees, and is also partially adaptive to unknown gradient norms, smoothness, and strong convexity. At the heart of our results is a novel parameter-free certificate for SGD step size choice, and a time-uniform concentration result that assumes no a-priori bounds on SGD iterates.
\end{abstract}

\section{Introduction}

\newcommand{\comparator}{\mathring{x}}

Stochastic convex optimization (SCO) is a cornerstone of both the theory and practice of machine learning. Consequently, there is intense interest in developing SCO algorithms that require little to no prior knowledge of the problem parameters, and hence little to no tuning~\cite{orabona2014simultaneous,mcmahan2017survey,kavis2019unixgrad,asi2019stochastic,loizou2021stochastic,vaswani2019painless}.
In this work we consider the fundamental problem of non-smooth 
SCO (in a potentially unbounded domain) and seek methods that are adaptive to a key problem parameter: the initial distance to optimality.

Current approaches for tackling this problem focus on the more general online learning problem of \emph{parameter-free regret minimization}  \cite{chaudhuri2009parameter,cutkosky2019artificial,cutkosky2017online,cutkosky2018black,kempka2019adaptive,mcmahan2014unconstrained,mhammedi2020lipschitz,orabona2016coin,orabona2017training,streeter2012no}, where the goal is to to obtain regret guarantees that are valid for comparators with arbitrary norms.
Research on parameter-free regret minimization has lead to practical algorithms
for stochastic optimization \cite{chen2022better,orabona2014simultaneous,orabona2017training}, methods that are able to adapt to many problem
parameters simultaneously  \cite{streeter2012no} and methods that
can work with any norm \cite{cutkosky2018black}.
In the basic Euclidean setting with 1-Lipschitz losses where only the initial distance
to optimality is unknown, there are essentially matching upper \cite{mcmahan2014unconstrained} and lower bounds \cite{orabona2013dimension}, showing that the best achievable parameter-free \emph{average} regret scales as 
\begin{flalign}\label{eq:parameter-free-regret}
	O\left( \| \comparator \| \sqrt{\frac{1}{T} \log\left( \frac{T \| \comparator \|^2}{\varepsilon^2} + 1 \right)} + \frac{\varepsilon}{T} \right)
\end{flalign}
where $T$ is the number of steps, 
$\| \comparator \|$ is the (Euclidean) comparator norm, and $\varepsilon > 0$ represents the (user-chosen) regret we will incur even if the comparator
norm is zero. This is larger by a logarithmic factor than the optimal average-regret when the comparator norm is known in advance.

Parameter-free regret bounds immediately translate into parameter-free SCO algorithms using online-to-batch conversion  \cite{hazan2016introduction}.
The expected optimality gap bound of the resulting algorithm is identical to~\eqref{eq:parameter-free-regret} when we replace $\comparator$ by $\xopt-x_0$, i.e., the difference between the optimum and the initial point. This bound is a logarithmic factor worse than what stochastic gradient descent (SGD) can achieve when we know the distance to optimality and use it to compute step sizes. While this logarithmic factor is unavoidable for regret minimization, it is unclear if it is necessary for SCO.

In this paper we show it is possible to obtain stronger parameter-free rates for SCO by moving beyond the regret minimization abstraction. In particular, for any $\varepsilon>0$ and $\delta\in(0,1)$, we a obtain probability $1-\delta$ optimality gap bounds of
\begin{equation*}
	O\prn*{
	\prn*{\frac{\norm{x_0-\xopt}}{\sqrt{T}} + \frac{\varepsilon}{T}} \log^{2} \prn*{\frac{1}{\delta}\log_+ \frac{T \norm{x_0-\xopt}}{\varepsilon}} },
\end{equation*}
which is better than any bound achievable by online-to-batch conversion. While 
replacing the logarithmic factor by a double-logarithmic factor may appear a small improvement, we consider it important due to the fundamental nature of the problem as well as the theoretical separation it establishes between parameter-free SCO and OCO. Such separations are rare in the literature; we are only aware of one prior example~\cite{hazan2014beyond}.

Our method also provides high probability guarantees on the suboptimality
gap. This resolves an open problem in parameter-free optimization; see \cite{orabona2014simultaneous} and \cite[\S7]{onlineLearningICMLtutorial}.
We are able to form high probability bounds because, unlike other
parameter-free SCO algorithms, we prove a strong localization guarantee: our output $\bar{x}$ satisfies $\norm{\bar{x}-\xopt} = O(\norm{x_0-\xopt})$, and key intermediate points satisfy a similar bound as well.
We suspect that such localization is difficult to establish with online-to-batch conversion, since online parameter-free algorithms  may need to let their iterates fluctuate wildly in order to handle difficult adversaries.

In addition to independence of $\norm{x_0-\xopt}$, our algorithm exhibits three additional forms of adaptivity. First, our algorithm has adaptivity to gradient norms on par with the best existing parameter-free result~\cite{cutkosky2018black}: the leading term of our bounds scales with a sum of squared observed gradient norms, and an a-priori gradient norm bound only affects low-order terms. Second, as a  consequence, in the smooth and noiseless case our algorithm exhibits a $\frac{\log\log T}{T}$ rate of convergence. Finally, via a simple restart scheme we obtain the optimal rate for strongly-convex stochastic problems (up to double-logarithmic factors), without knowledge of the strong-convexity parameter.

On a technical level, our development differs significantly from prior  parameter-free optimization methods. While online methods rely on advanced tools such as coin betting \cite{orabona2016coin}, and online Newton steps \cite{cutkosky2018black}, %
our approach is essentially a careful scheme for correctly setting the step size of SGD.
Underlying our algorithm is a parameter-free certificate for SGD, which implies both localization and optimality gap bounds. The certificate takes the form of an implicit equation over the SGD step size, which we solve via bisection on the logarithm of the step size. To obtain high-probability bounds, we develop a time-uniform empirical-Berstein-type concentration bound independent of any a-priori assumptions on the iterate norms. 
Given the ubiquity of SGD in practice and in the classroom, our insights
on how choose its step size may be of independent interest.

\paragraph{Paper organization.} In the following subsections we review additional related work, as well as the problem setup and notation. Section~\ref{sec:algorithm} develops our parameter-free step size certificate. 
Section~\ref{sec:deterministic} presents our algorithm and its analysis in the  noiseless regime. Section~\ref{sec:stochastic} lifts the analysis to the stochastic setting, proving our main result on parameter-free SCO. Finally, 
Section~\ref{sec:adaptivity} shows how our method adapts to smoothness and (via restarts) to strong convexity.

\subsection{Additional related work}

\paragraph{Parameter-free methods from deterministic optimization.}
The literature on noiseless optimization also offers a rich variety of parameter-free algorithms. In the smooth setting, the Armijo rule \cite{armijo1966minimization} is a standard technique for choosing step sizes for gradient descent. 
Using variants of this idea combined with acceleration, 
achieves essentially optimal and parameter-free rates of convergence~\cite{beck2009fast}. The Polyak step size rule~\cite{polyak1987} simultaneously achieves optimal rates for smooth, non-smooth and strongly-convex optimization~\cite{hazan2019revisiting}, but requires knowledge of the optimal function value. This requirement can be relaxed, making the Polyak method parameter-free, but at the cost of a multiplicative logarithmic factor to its bound~\cite{hazan2019revisiting}. Consequently, non-smooth parameter-free deterministic optimization appears to be as hard as SCO. Multiple works generalize line-search and the Polyak method to the stochastic setting \cite{rolinek2018l4,berrada2020training,asi2019stochastic,loizou2021stochastic,vaswani2019painless,zhang2020statistical,byrd2012sample}, but do not obtain  parameter-free rates in the sense we consider here.

\paragraph{Limitations of online-to-batch conversion.}
To the best of our knowledge, the only previous example of an SCO rate that is provably unachievable by online to batch conversion of a (uniform) regret bound occurs for strongly-convex optimization. Specifically, any online strongly-convex optimization algorithm must have logarithmic regret (implying suboptimality $(\log T)/T$ via online to batch conversion)~\cite{takimoto2000minimax,hazan2014beyond}, while  \citet{hazan2014beyond} 
and others~\cite{juditsky2014primal,ghadimi2012optimal,shamir2013stochastic} have achieved the optimal $1/T$ rate for stochastic strongly-convex optimization. The variant of our algorithm in \Cref{sec:adaptivity-sc} is based on the Epoch-SGD algorithm of \cite{hazan2014beyond}, and simultantiously breaks both regret minization barriers, achieving optimality gap $(\log \log T)/T$ for parameter-free strongly-convex stochastic optimization with high 
probability.

\paragraph{Grid search.} 
In practice, the standard technique for selecting the step size of SGD  (and hyperparameters more broadly) is grid search \cite{pedregosa2011scikit,feurer2019hyperparameter}.
This typically consists of testing all step sizes on a geometrically spaced grid and choosing the one with the best performance on a held out set. Compared to our method, such grid search is computationally wasteful, as it tests exponentially more steps sizes than we do. Moreover, in the context of parameter-free SCO, proving guarantees for grid search  is surprisingly difficult, since it is unclear how to bound the objective value estimation error for points that may be arbitrarily far apart.
\subsection{Problem setup and notation}
\newcommand{\oracle}{\mathcal{O}}
\newcommand{\projX}{\Pi_{\xset}}
Let us briefly review the standard SCO setup, building up our notation along the way.
Our goal is to minimize a convex objective function $f:\xset \to \R$ defined on a closed an convex set $\xset\subseteq\R^d$ (our results hold for $\xset=\R^d$ as well). We let $\xopt$ denote a fixed minimizer of $f$, i.e., such that $f(\xopt) \le f(x)$ for all $x\in\xset$, implicitly assuming such point exists (see \Cref{sec:relaxed-xopt} for further discussion of this assumption).
We assume that our only access to $f$ is via a stochastic gradient oracle $\oracle$ that, upon receiving query point $x$, returns a vector $\oracle(x)\in\R^d$ that is a subgradient of $f$ in expectation, i.e., $\Ex{ \oracle(x) | x} \in \partial f(x)$. With slight abuse of notation, we write $\grad f(x) \defeq \Ex{ \oracle(x) | x}$, corresponding to the gradient of $f$ when it is differentiable and a particular subgradient otherwise. We interchangeably use \emph{exact gradients}, \emph{noiseless}, and \emph{deterministic} to refer to the regime where $ \oracle(x) = \grad f(x)$ with probability 1.

Our development revolves around the classical fixed step size stochastic gradient descent (SGD) algorithm. Given step size $\eta$ and initialization $x_0$, SGD iterates
\begin{equation*}
	x_{i+1}(\eta) = \projX\prn[\big]{x_{i}(\eta) - \eta g_i(\eta)}
	~~\mbox{where}~~g_i(\eta) \defeq \oracle(x_i(\eta)),
\end{equation*}
and $\projX$ is the Euclidean projection onto $\xset$; we intentionally  feature the $\eta$ dependence of $x_i(\eta)$ and $g_i(\eta)$ prominently. We define the following quantities associated with the SGD iterates. First, we write the distance to $\xopt$ and its running maximum as
\begin{equation*}
	\d[t](\eta) \defeq \norm{x_t(\eta)-\xopt}~~\mbox{and}~~
	\dbar[t](\eta) \defeq \max_{i\le t}\norm{x_i(\eta)-\xopt}.
\end{equation*} 
Replacing $\xopt$ with $x_0$ in the above definitions, we write
\begin{equation}\label{eq:rbar-def}
	\r[t](\eta) \defeq \norm{x_0 - x_t(\eta)}~~\mbox{and}~~
	\rbar[t](\eta) \defeq \max_{i \le t}\norm{x_0  - x_i(\eta)}.
\end{equation}
Finally, we denote the running sum of squared gradient norms and gradient oracle error by
\begin{equation}\label{eq:G-def}
	\G[t](\eta) \defeq \sum_{i<t} \norm{g_i(\eta)}^2
	~~\mbox{and}~~
	\gdiff[i] \defeq g_i(\eta) - \grad f(x_i(\eta)).
\end{equation}

\paragraph{Additional notational conventions.} Throughout, $\norm{\cdot}$  denotes the Euclidean norm. We use $\log$ to denote the base 2 logarithm, 
and write $\log_+(x) \defeq \max\{2,\log(x)\}$ to simplify $O(\cdot)$ notation. For any particular value of $\eta$, the quantities $x_i(\eta)$, $g_i(\eta)$, etc.\ always refer to a \emph{single} realization of the random process they represent.

\section{A parameter-free step-size selection criterion for SGD}\label{sec:algorithm}
In this section we present the key component of our development: a computable certificate for the efficiency of a candidate SGD step size. For ease of exposition, in this section we restrict some of our arguments to the exact gradient setting, but emphasize that they ultimately translate to high-probability bounds in the stochastic setting.

Consider the noiseless setting with step size $\eta$, iterates $x_0, x_1(\eta),\ldots,x_T(\eta)$ and gradients $g_0$, $g_1(\eta)$, $\ldots, g_{T-1}(\eta)$. It is well-known \cite{onlineLearningICMLtutorial} that if $\eta$ satisfies
\begin{equation*}
	\eta = \bt_{\mathrm{ideal}}(\eta) ~~\mbox{where}~~
	\bt_{\mathrm{ideal}}(\eta) \defeq \frac{\norm{x_0 - \xopt}}{\sqrt{\sum_{i<T} \norm{g_i(\eta)}^2}} = \frac{\d[0]}{\sqrt{\G(\eta)}}
\end{equation*}
then the iterate average $\bar{x}(\eta) = \frac{1}{T}\sum_{i<T} x_i(\eta)$ satisfies the optimal error bound $f(\bar{x})-f(\xopt) \le d_0 \sqrt{\G(\eta)} / T$, scaling as $O( d_0 L / \sqrt{T})$ when $\norm{g_i(\eta)} \le L$ for all $i$. However, the quantity $\bt_{\mathrm{ideal}}$, which we call the ``ideal step size'' is not computable even in hindsight (when $\G(\eta)$ is available), since the parameter $\d[0] = \norm{x_0 - \xopt}$ is unknown.

Our key proposal is to approximate the distance to the optimum $\d[0]$ with a computable proxy: the maximum distance traveled by the algorithm, $\rbar(\eta) \defeq \max_{i\le T}\norm{x_0 - x_i(\eta)}$. We consider step sizes that (approximately) satisfy
\begin{equation}\label{eq:step-size-implicit-equation}
	\eta = \bt(\eta)
	~~\mbox{where}~~
	\bt(\eta) \defeq \frac{\rbar(\eta)}{\sqrt{\Gcoef \G(\eta) + \Ccoef}}
\end{equation}
for nonnegative damping parameters $\Gcoef$ and $\Ccoef$; in the exact gradient setting we can set $\Gcoef$ to any number $>1$ and $\Ccoef=0$, while the stochastic setting requires scaling $\Gcoef$ and $\Ccoef$ roughly as $\mathsf{poly}(\log\log T)$. Intuitively, $\rbar$ approximates $\d[0]$ since the SGD iterates should converge to $\xopt$ and therefore $\norm{x_0-x_T(\eta)}$ should be similar to $\norm{x_0-\xopt}$. 
However, in non-smooth optimization, convergence to $\xopt$ can be arbitrarily slow. We nevertheless prove that, when $\eta \le \bt(\eta)$, we have $\rbar(\eta) = O(\d[0])$ (\Cref{lem:distance-bound} below). With this result  and a refined SGD error bound (\Cref{lem:gm-error-bound-r} below), we show that (with exact gradients) any $\eta$ satisfying criterion~\eqref{eq:step-size-implicit-equation} recovers the optimal error bound.
\begin{prop}\label{prop:implicit-result}
	In the noiseless setting, 
	any step size $\eta > 0$ satisfying~\eqref{eq:step-size-implicit-equation} with $\Gcoef >1 $ and $\Ccoef=0$ produces $\bar{x} \defeq \frac{1}{T} \sum_{i< T} x_i (\eta)$ such that $\| x_\star - \bar{x} \| \le \frac{2\Gcoef}{\Gcoef - 1} \| x_\star - x_0 \|$ and
	\begin{equation*}
		f(\bar{x}) - f(\xopt)
		\le  \frac{\Gcoef^{3/2}}{\Gcoef -1} \cdot \frac{\d[0] \sqrt{\G[T](\eta)}}{T} = O\prn*{\frac{\d[0] \sqrt{\G[T](\eta)}}{T}}.
	\end{equation*}
\end{prop}

Before proving \Cref{prop:implicit-result}, let us briefly discuss its algorithmic implications. Since the function $\bt(\cdot)$ is computable (at the cost of $T$ gradient queries) without a-priori assumptions on $\d[0]$, we have reduced parameter-free optimization to solving the one-dimensional implicit equation~\eqref{eq:step-size-implicit-equation}. However, the function $\bt$ might be discontinuous and an exact solution to the implicit equation might not even exist. Nevertheless, in the next section we show that finding an interval $[\eta, 2\eta]$ in which  $h\mapsto\bt(h)-h$ changes  sign, produces nearly the same error certificates at an interval edge. Since such interval is readily found via bisection, this forms the basis of a working parameter-free step size tuner. We leave the details to \Cref{sec:deterministic} and for the remainder of this section prove~\Cref{prop:implicit-result}.

\subsection{Proof of \Cref{prop:implicit-result}}
The proof of \Cref{prop:implicit-result} hinges on two lemmas. The first is a variant of the standard SGD error bound (recall that $\gdiff[i](\eta) = g_i(\eta) - \grad f(x_i(\eta))$ is zero in the noiseless setting).

\begin{lem}\label{lem:gm-error-bound-r}
For $T\in \N$, $\bar{x} \defeq \frac{1}{T} \sum_{i< T} x_i (\eta)$, and $\eta>0$,  we have
\begin{equation}\label{eq:gm-error-bound-r}
f(\bar{x}) - f(\xopt) \le 
\frac{1}{T} \sum_{i<T} f(x_i(\eta)) - f(\xopt) \le 
\frac{\rbar(\eta)\d[0]}{\eta T} + \frac{\eta \G[T](\eta)}{2T} +  \frac{1}{T} \sum_{i < T}  \inner{\gdiff[i](\eta)}{x_\star - x_i(\eta)}.
\end{equation}
\end{lem}

\begin{proof}
Since $\eta$ is fixed throughout this proof, we streamline notation by dropping it from $x_i, g_i, \gdiff[i]$ and $\G[i]$. 
By convexity and the definition of $\gdiff[i]$,
\begin{flalign}\label{eq:jensen-expanded}
\frac{1}{T} \sum_{i<T} f(x_i) - f(\xopt)  \le \frac{1}{T} \sum_{i < T}  \inner{\grad f(x_i)}{x_i-\xopt}   = \frac{1}{T} \sum_{i < T}  {\inner{ g_i}{x_i-\xopt} +  \inner{\gdiff[i]}{x_\star - x_i}}.
\end{flalign}
From $x_{i+1} = \projX(x_i- \eta g_i)$ we can derive the standard subgradient method inequality
\begin{equation}\label{eq:standard-sgm-inequality}
	\d[i+1]^2 \le \norm{x_i - \eta  g_i - \xopt}^2  = \d[i]^2 - 2 \eta \inner{ g_i}{x_i-\xopt}  + \eta^2 \|  g_i \|^2
\end{equation}
for all $\eta \ge 0$ and $i = 0, \dots, T-1$. Rearranging and summing over $i<T$ gives
\begin{equation*}
\sum_{i < T}  \inner{ g_i}{x_i-\xopt} \le \frac{\d[0]^2-\d[T]^2}{2\eta} + \frac{\eta \G[T]}{2} = \frac{(\d[0]-\d[T])(\d[0]+\d[T])}{2 \eta} + \frac{\eta \G[T]}{2} \overle{(\star)} \frac{\rbar[T] \cdot 2\d[0]}{2 \eta} + \frac{\eta \G[T]}{2}.
\end{equation*}
The inequality $(\star)$ is where our proof deviates from the textbook derivation \cite[Theorem 2.13.]{orabona2021modern}; 
it holds because $\d[0]-\d[T] \le \r[T]$ due to the triangle inequality, and either $\d[T] \le \d[0]$ holds or $\d[0]^2-\d[T]^2 \le 0 \le 2\d[0] \rbar[T]$. Substituting into \eqref{eq:jensen-expanded} and applying $f(\bar{x}) \le\frac{1}{T} \sum_{i<T} f(x_i)$ by Jensens's inequality gives \eqref{eq:gm-error-bound-r}.
\end{proof}

The second lemma shows that for $\eta$ satisfying $\eta\le \bt(\eta)$ is guaranteed to produce iterates that do not wander too far from $\xopt$.
This is our basic localization guarantee.

\begin{lem}\label{lem:distance-bound}
	In the noiseless setting with $\Gcoef > 1$ and $\eta > 0$, if $\eta \le \bt(\eta)$ then we have  $\dbar(\eta) \le \frac{\Gcoef +1 }{\Gcoef -1} \d[0]$
		and 
		$\rbar(\eta) \le \frac{2\Gcoef}{\Gcoef - 1} \d[0]$.
\end{lem}

\begin{proof}
	This proof once more drops the explicit dependence on $\eta$. 
	Summing the inequality~\eqref{eq:standard-sgm-inequality} over $i<t$ and noting that that $\inner{g_i}{x_i - \xopt} \ge f(x_i)  - f(\xopt) \ge 0$ due to convexity and the noiseless setting, we have $\d[t]^2 \le \d[0]^2 + \eta^2 \G[t]$
	for every $t$. Maximizing over $t\le T$ and substituting $\eta \le \bt(\eta) \le \rbar / \sqrt{\Gcoef \G[T]}$ yields
	\begin{equation}\label{eq:deterministic-quadratic-ineq}
		\dbar[T]^2 \le \d[0]^2 + \eta^2 \G[T] \le \d[0]^2 + \frac{1}{\Gcoef}\rbar[T]^2 \overle{(\star)}   \d[0]^2 + \frac{1}{\Gcoef}(\dbar[T] + \d[0])^2,
	\end{equation}
	where $(\star)$ follows from the triangle inequality:  $\rbar[T]=\r[t] \le \d[t] + \d[0] \le \dbar[T] + \d[0]$ for some $t\le T$. 
	Rearranging yields
$\prn*{\dbar[T] - \frac{1}{\Gcoef - 1}\d[0]}^2 \le \frac{\Gcoef^2}{(\Gcoef-1)^2} \d[0]^2,$
	and therefore $\dbar[T] \le 
	\frac{\Gcoef+1}{\Gcoef-1}\d[0]$ as required. The bound $\rbar[T] \le \frac{2\Gcoef}{\Gcoef-1} \d[0]$ follows from substituting $\dbar[T] \le \frac{\Gcoef+1}{\Gcoef-1}\d[0]$ into $\rbar[T] \le \dbar[T] + \d[0]$.
\end{proof}

\Cref{prop:implicit-result} follows from substituting $\eta = \rbar(\eta) / \sqrt{\Gcoef \G}$ into bound~\eqref{eq:gm-error-bound-r} yielding $f(\bar{x})-f(\xopt) \le \frac{\d[0]\sqrt{\Gcoef\G(\eta)}}{T} +
 \frac{\rbar(\eta)\sqrt{\G(\eta)}}{2\sqrt{\Gcoef}T}$, 
and using $\rbar(\eta) \le \frac{2\Gcoef}{\Gcoef-1} \d[0]$ from \Cref{lem:distance-bound}. \hfill$\blacksquare$

\section{Algorithm description, and analysis for exact gradients}\label{sec:deterministic}

In this section we turn the step-size selection criterion presented in the previous section into a complete algorithm (\Cref{alg:main})---valid for stochastic as well as exact gradients---and analyze it in the simpler setting of exact gradients, deferring the stochastic case to the following section. 
Our algorithm consists of a core log-scale bisection subroutine ($\rootfinding$) 
coupled with an outer loop that acts as an aggressive doubling scheme on the upper limit of the bisection. We describe and analyze the two components \Cref{ssec:deterministic-bisection,ssec:deterministic-doubling}, respectively. Then, in \Cref{ssec:deterministic-final}, we put these results together and obtain parameter-free rates in the exact gradient setting.

\begin{algorithm2e}[t]
	\setstretch{1.1}
	\caption{Parameter-free SGD step size tuning}
	\label{alg:main}
	\LinesNumbered
	\DontPrintSemicolon
	\KwInput{
		Initial step size $\etainit > 0$, total gradient budget $B \in \N$, %
		 constants $\{\GcoefK, \CcoefK\}$ 
		
		\Comment{In the deterministic case, $\GcoefK=3$ (or any constant $>1$) and $\CcoefK=0$; in the stochastic case see eq.~\eqref{eq:coef-k-setting}}
	}
	\SetKwProg{Fn}{function}{}{}
	\For{$k = 2, 4, 8, 16, \dots$ }{ 
		\lIf{$k > B/4$}{ 
			\Return $x_0 $\Comment*[f]{Only happens in the edge case $B=O\prn*{\log \log \frac{\norm{\xopt - x_0}}{\etainit \sqrt{ \GcoefK[0] \norm{g_0}^2 + \CcoefK[0]} }}$}\label{line:main-alg-edge-case-termination}
		}
		$T_k \gets \floor*{\frac{B}{2 k}}$\;
		$\etaout \gets \rootfinding(\etainit, 2^{2^{k}} \etainit; T_k, \GcoefK, \CcoefK)$

		\lIf{$\etaout < \infty$}{\Return $\frac{1}{T_k} \sum_{i< T_k} x_i(\etaout)$}\label{line:main-alg-normal-termination}
	}

	\Fn{$\rootfinding(\etalo,\etahi;T,\Gcoef,\Ccoef)$}{
		$\bt \defeq \eta \mapsto \rbar(\eta) / \sqrt{\Gcoef \G(\eta) + \Ccoef}$
		\quad\Comment*[l]{Bisection target ($\rbar$ and $\G$ defined in eqs.~\eqref{eq:rbar-def} and~\eqref{eq:G-def})}\label{line:bisection-target-def}

		\lIf{$\etahi \le  \bt(\etahi)$}{\Return $\infty$  
			\notarxiv{\hspace{24pt}}
			\arxiv{\hspace{8pt}}
			\Comment*[h]{$\etahi$ is too low and should be increased}}\label{line:bisection-etahi-check}
		
		\lIf{$\etalo >  \bt(\etalo)$}{\Return $\etalo$ 
			\notarxiv{\hspace{23.7pt}}
			\arxiv{\hspace{7.7pt}} 
			\Comment*[h]{$\etalo$ is sufficient (assuming it is very small)}}\label{line:bisection-etalo-check}

		\While(\notarxiv{\hspace{72pt}}\arxiv{\hspace{66.5pt}}\Comment*[h]{Invariant: $\etalo < \etahi$, $\etalo \le \bt(\etalo)$, $\etahi > \bt(\etahi)$}){$\etahi > 2\etalo$}{\label{line:while-loop-termination}
			$\etamid \gets \sqrt{\etalo \etahi}$\;

			\leIf{$\etamid \le \bt(\etamid)$}{$\etalo \gets \etamid$}{$\etahi \gets \etamid$}\label{line:bisection-etamid-check}

		}

		\leIf{$\rbar(\etahi) \le \rbar(\etalo) \frac{\bt(\etahi)}{\etahi}$}{\Return $\etahi$}{\Return $\etalo$}\label{line:bisection-normal-termination}
		
	}
	
\end{algorithm2e}

\subsection{Bisection subroutine}\label{ssec:deterministic-bisection}

Let us describe the  $\rootfinding$ subroutine of \Cref{alg:main}. Its input is an initial interval $[\etalo,\etahi]$, SGD iteration number $T$ and damping parameters $(\Gcoef, \Ccoef)$ for defining the bisection target $\bt(\eta) = \rbar(\eta)/\sqrt{\Gcoef \G(\eta) + \Ccoef}$. After testing that $\etalo \le \bt(\etalo)$ and $\etahi > \bt(\etahi)$ (and handling the edge cases where this does not hold), we iteratively shrink the interval $[\etalo,\etahi]$ by replacing one of its edges with $\sqrt{\etalo\etahi}$ while maintaining the invariant $\etalo \le \bt(\etalo)$ and $\etahi > \bt(\etahi)$.\footnote{
	Our choice of $\etamid = \sqrt{\etalo\etahi}$, which corresponds to  standard bisection on a \emph{log-scale}, is crucial:  the standard choice  $\etamid=\frac{1}{2}(\etalo+\etahi)$ would result in a logarithmic rather than double-logarithmic number of bisection steps.
}  
The iterations stop  when $\etahi / \etalo \le 2$. Since each iteration halves $\log \frac{\etahi}{\etalo}$, the overall iteration number is double-logarithmic in the ratio of the input $\etahi$ and $\etalo$. Specifically, if the input interval satisfies $\etahi / \etalo = 2^{2^k}$ for $k\in \N$ (which it does in  \Cref{alg:main}), then $\rootfinding$ performs exactly $k = \log\log \frac{\etahi}{\etalo}$ bisection steps. Consequently, the overall oracle complexity of the the subroutine is $O(T \log \log_+ \frac{\etahi}{\etalo})$.  

We now focus on the end of the bisection procedure and explain the choice of output in \cref{line:bisection-normal-termination}.
When the bisection loop is complete, we obtain a relatively narrow interval $[\etalofinal, \etahifinal]$ in which $\bt(\eta)-\eta$ is guaranteed to change its sign. When $\bt$ is continuous, this implies that some $\eta \in [\etalofinal, \etahifinal]$ solves $\eta=\bt(\eta)$ and therefore has a good error bound by~\Cref{prop:implicit-result}. However, $\bt$ is not necessarily continuous. To explain why the bisection still outputs a good value of $\eta$, first note that \Cref{prop:implicit-result} continues to hold (with a slightly worse constant factor) even when $\eta$ only approximately solves $\eta=\bt(\eta)$, e.g., when
\notarxiv{$
\frac{\rbar(\eta)}{2\sqrt{\Gcoef \G(\eta) + \Ccoef}} = \tfrac{1}{2} \bt(\eta) \le \eta \le \bt(\eta) = \frac{\rbar(\eta)}{\sqrt{\Gcoef \G(\eta) + \Ccoef}}.
$}
\arxiv{
\begin{equation*}
	\frac{\rbar(\eta)}{2\sqrt{\Gcoef \G(\eta) + \Ccoef}} = \tfrac{1}{2} \bt(\eta) \le \eta \le \bt(\eta) = \frac{\rbar(\eta)}{\sqrt{\Gcoef \G(\eta) + \Ccoef}}.
\end{equation*}
}
The following lemma shows that the output of $\rootfinding$  in fact satisfies a similar bound. 
See Appendix~\ref{app:proof-of:lem:output-properties} for the (easy) proof, and note the lemma also holds in the stochastic case.

\begin{lem}\label{lem:output-properties}
	Let $\etalo, \etahi, \alpha, \beta > 0$, and $T \in \N$.
	If $\rootfinding(\etalo,\etahi;T,\Gcoef,\Ccoef)$ terminates in \cref{line:bisection-normal-termination} with final interval $[\etalofinal,\etahifinal]$ and returns $\etaout$, then 
	\begin{flalign}\label{eq:output-property} 
		\frac{\rbar(\etaout)}{2 \sqrt{\Gcoef \G(\etahifinal) + \Ccoef}} \le \etaout \le \frac{\rbar(\etalofinal)}{\sqrt{\Gcoef \G(\etaout) + \Ccoef}}.
	\end{flalign}
	Moreover, we have $\rbar(\etaout) \le \rbar(\etalofinal)$ and $\sqrt{\Gcoef \G(\etaout) + \Ccoef} \le 2  \sqrt{\Gcoef \G(\etahifinal) + \Ccoef}$.
\end{lem}

We now combine Lemmas \ref{lem:gm-error-bound-r}, \ref{lem:distance-bound} and \ref{lem:output-properties} to show an error bound for GD with the $\eta$ selected by $\rootfinding$. The proof of Proposition~\ref{prop:subopt-bound}
appears in Appendix~\ref{app:proof-of:subopt-bound}.

\begin{prop}\label{prop:subopt-bound}
	In the noiseless setting,
	let $\etaout = \rootfinding(\etalo,\etahi;T,\Gcoef,\Ccoef)$ for $\Gcoef>1$, $\etalo, \etahi > 0$, $\Ccoef \ge 0$, and $T \in \N$, assume that $\etahi > \bt(\etahi)$, and let $\bar{x} = \frac{1}{T} \sum_{i < T} x_i(\etaout)$. If $\etalo \le  \bt(\etalo)$, then
	\begin{equation*}\label{eq:deterministic-subopt-main}
		\norm{\bar{x} - x_0} \le \frac{2\Gcoef}{\Gcoef-1}\d[0]
		~~\mbox{and}~~
		f(\bar{x}(\etaout)) - f(\xopt) \le  \frac{2\Gcoef }{\Gcoef-1} \cdot \frac{\d[0] \sqrt{ \Gcoef \G(\eta') + \Ccoef}}{T}
	\end{equation*} 
	for some $\eta' \in [\etaout, 2\etaout]$. If instead $\etalo > \bt(\etalo)$,  then $\etaout = \etalo$ and
	\begin{equation*}\label{eq:deterministic-subopt-edge}
		\norm{\bar{x} - x_0} \le  \etalo \sqrt{\Gcoef \G(\etalo) + \Ccoef}
		\mbox{~~and~~}
		f(\bar{x}) - f(\xopt) \le \frac{ \d[0] \sqrt{\Gcoef \G(\etalo) + \Ccoef} + \etalo\G(\etalo) }{T}.
	\end{equation*}
\end{prop}

Let us briefly summarize our findings so far. 
When \rootfinding{} returns at \cref{line:bisection-normal-termination}, \Cref{prop:subopt-bound} provide a bound similar to the best achievable  when $\d[0]$ is known, with only a factor $O(\log\log\frac{\etahi}{\etalo})$  complexity increase. If instead the lower limit of the bisection is invalid, i.e., $\etalo > \bt(\etalo)$, our bound becomes the optimal rate plus a term proportional to $\etalo$. Therefore, by picking a very small value of $\etalo$ we can ensure a good error bound in that case as well.\footnote{ 
Moreover, when $\etalo$ is very small we  expect $\etalo \le \bt(\etalo)$ to hold, since then (intuitively) $g_i(\etalo) \approx g_0$ for all $i$, which implies $\bt(\etalo) \approx \etalo \sqrt{T/\Gcoef} > \etalo$ for sufficiently large $T$.
}
However, if the upper limit of the bisection is invalid, i.e., $\etahi \le \bt(\etahi)$, then \rootfinding{} fails, returning $\etaout=\infty$. We address this issue next.

\subsection{Doubling scheme for upper bisection limit}\label{ssec:deterministic-doubling}

\Cref{alg:main} iteratively calls \rootfinding{} with upper bisection limits $\etahi$ of the form $2^{2^k} \etainit$ (for doubling values of $k$) until the bisection returns $\etaout < \infty$, i.e., until $\etahi > \bt(\etahi)$. To ensure the overall number of gradient queries never exceeds the budget $B$, for every $k$ the algorithm also adjusts the SGD complexity $T$. In the stochastic case, the parameters $\Gcoef$ and $\Ccoef$ also increase with $k$ in order to enforce a union bound over an increasing number of SGD sample paths.

Intuitively, the bisection should succeed once $\etahi > \d[0]/\norm{g_0}$ since this is always an upper bound on the ideal step size $\bt_{\mathrm{ideal}}$. Even though we do not know $\d[0]$ and therefore cannot set $\etahi$ a-priori,\footnote{
	If an upper bound  $D \ge \d[0]$ is available (e.g., the domain diameter) then we may use it instead of a doubling scheme by directly fixing $k$ to be $\log\log \frac{D}{\etainit \norm{g_0}}$. However, this can improve our error bounds by at most a constant factor.} 
 \Cref{alg:main} will reach such $\etahi$ when $k$ is roughly $\log\log \frac{\d[0]}{\etainit\norm{g_0}}$. Lemma~\ref{lem:eta-max}, whose proof appears in Appendix~\ref{app:proof-of:lem:eta-max}, provides a rigorous version of our intuitive reasoning.
	
\newcommand{\kout}{k_{\mathrm{o}}}
\begin{lem}\label{lem:eta-max}
	In the noiseless setting with $\Gcoef > 1$, $\beta\ge0$ and any $T \in \N$, if $\eta > \etamax \defeq \frac{2\Gcoef}{\Gcoef-1}\cdot \frac{ \d[0]}{ \sqrt{\Gcoef \norm{g_0}^2 + \Ccoef}}$ then 
		 $\eta > \bt(\eta)$.
	Consequently, when \Cref{alg:main} terminates, $k \le 2 \log \log_+ \frac{\etamax}{\etainit}$.
\end{lem}

Note that we do not perform a similar doubling scheme to search for an $\etalo$ value satisfying $\etalo \le \bt(\etalo)$, since there is no $\etalo$ for which such bound is guaranteed to hold: when $\d[0] = 0$ it is possible to have $\bt(\eta)=0$ for all $\eta$. More broadly, parameter-free methods must suffer some non-zero error when  $\d[0]=0$~\cite{orabona2013dimension}, and for our method such term only appears under the condition $\etalo > \bt(\etalo)$ of \Cref{prop:subopt-bound}, strongly suggesting that we cannot always force $\etalo \le \bt(\etalo)$ to hold.

\subsection{Error guarantees for exact gradients}\label{ssec:deterministic-final}

With \Cref{alg:main} explained and \Cref{prop:implicit-result} and \Cref{lem:eta-max} in place, we are ready to state the parameter-free convergence guarantee in the exact gradient setting. For simplicity of exposition, we fix $\Gcoef=3$ and $\Ccoef=0$, but note that any $\Gcoef > 1$ yields a similar guarantee.

\begin{theorem}\label{thm:deterministic-main}
	In the noiseless setting \Cref{alg:main}, with parameters $\GcoefK = 3$, $\CcoefK =0$, $\etainit > 0$, $B\in \N$, and $x_0\in\R^d$, performs at most $B$ subgradient queries and returns $\bar{x} = \frac{1}{T}\sum_{i< T}x_i(\eta)\in \R^d$ for some $\eta \ge \etainit$ and integer $T$ satisfying
	\begin{equation}\label{eq:determinisitic-T-lb}
		T \ge \max\crl*{ \frac{B}{12 \log \log_+ \frac{\norm{x_0 - \xopt}}{\etainit \norm{g_0}}}, 1 }
	\end{equation}
	such that either
	\begin{equation}\label{eq:deterministic-main-bound-normal-case}
		\norm{\bar{x}-\xopt} \le 4\norm{x_0 - \xopt}
		~~\mbox{and}~~
		f(\bar{x}) - f(\xopt) \le \sqrt{27}\,\frac{ \norm{x_0 - \xopt}\sqrt{\G(\eta')}}{T}
	\end{equation}
	for some $\eta' \in [\eta, 2\eta]$, or $\eta = \etainit$ and
	\begin{equation}\label{eq:deterministic-main-bound-edge-case}
		\norm{\bar{x}-x_0} \le  \etainit \sqrt{3\G(\etainit)}
		~~\mbox{and}~~
		f(\bar{x}) - f(\xopt) \le 2\,\frac{\etainit \G(\etainit)}{T}.
	\end{equation}
\end{theorem}

The proof of \Cref{thm:deterministic-main} appears in Section~\ref{app:proof-of:thm:deterministic-main}.
Let us briefly compare the bounds in \Cref{thm:deterministic-main} to our guarantees for a solution to $\eta=\bt(\eta)$ shown in \Cref{prop:implicit-result}. The ``typical case'' bound~\eqref{eq:deterministic-main-bound-normal-case} is similar to the error bound of \Cref{prop:implicit-result} with only two notable differences beyond a slightly larger constant factor. First, by eq.~\eqref{eq:determinisitic-T-lb}, the value of $T$ in \Cref{thm:deterministic-main} is smaller than the total complexity budget by a double-logarithmic factor; this is the cost of performing a bisection instead of
assuming we start with a solution to the implicit equation. Second, the term $\G(\eta')$ in the RHS of~\eqref{eq:deterministic-main-bound-normal-case} is computed at $\eta'$ that is possibly different than the $\eta$ for which we prove the error bound. 

While bounding the error of SGD with step size $\eta$ using the gradients observed by SGD with a different step size $\eta'$ is unconventional,  our resulting bounds appear to be as useful as their more conventional counterparts. First, note that $\eta$ and $\eta'$ are within a factor of 2 of each other, and we can bring this factor arbitrarily close to 1 by running more bisection steps. Second, despite the difference in $\eta$, we can still use our lower bounds to obtain (up to double-logarithmic factors) a $1/T$ rate of convergence for smooth problems with unknown smoothness (see \Cref{sec:adaptivity-smooth}); this is the hallmark of error bounds that scale with $\sqrt{\G}$. Moreover, the different $\eta'$ issue disappears when we assume $f$ is $L$-Lipschitz and uniformly bound $\G(\eta)$ by $L^2 T$, as we do below.

We conclude this section with a particularly useful choice of $\etainit$. Let $\reps$ be a putative lower bound on $\d[0]$ and take 
\[
\etainit = \frac{\reps}{\norm{g_0} B}.
\]
Apply \Cref{alg:main} as described in \Cref{thm:deterministic-main}, obtaining $\bar{x}$ and $\eta$. To handle the case $\eta=\etainit$, set 
\[
z \defeq \begin{cases}
x_0 & \eta = \etainit \text{ and } \|g_0 \| \le \frac{\sqrt{G_T( \eta_{\varepsilon})}}{T} \\
\bar{x} & \text{otherwise}.
\end{cases} 
\] 
When $z=\bar{x}$, \Cref{thm:deterministic-main} guarantees that either $f(z)-f(\xopt) \le \sqrt{27} \frac{\d[0] \sqrt{\G(\eta')}}{T}$ (when $\eta > \etainit$) or $f(z)-f(\xopt) \le \frac{2\etainit \G(\etainit)}{T} = \frac{2\reps \G(\etainit)}{\norm{g_0} BT} \le \frac{2\reps \sqrt{\G(\etainit)}}{T}$ (when $\eta=\etainit$ and $\norm{g_0} > \frac{\sqrt{G_T( \eta_{\varepsilon})}}{T}$). 
When $z=x_0$  we have that $f(z) - f(\xopt) \le \norm{g_0} \d[0] \le \frac{\reps \sqrt{\G(\etainit)}}{T}$ by convexity and the fact that $\norm{g_0} \le \frac{\sqrt{G_T( \eta_{\varepsilon})}}{T}$ since $z=x_0$. In conclusion, we always have
\begin{equation*}
	f(z)-f(\xopt) \le \sqrt{27}\, \frac{\d[0] \sqrt{G_T( \eta')} + \reps\sqrt{G_T( \etainit )} }{T}.
\end{equation*}
Assuming that $f$ is $L$-Lipschitz so that $\G[T](\eta) \le L^2 T$ for all $\eta$ and using $T \ge \frac{B}{12 \log \log_+ \frac{B \d[0]}{\reps}}$ by \Cref{thm:deterministic-main}, we get that
\[
	f(z)-f(\xopt) \le 18 \sqrt{\log\log_+ \prn*{\frac{B d_0}{\reps}} } \frac{L (\d[0] + \reps)}{\sqrt{B}}.
\]
\section{Analysis for stochastic gradients}\label{sec:stochastic}

In this section, we extend the analysis of \Cref{alg:main} to the stochastic setting, using the following simple strategy: we define a ``good event'' under which the noiseless analysis goes through essentially unchanged (\Cref{ssec:stochastic-good-event}), and show that this event occurs with high probability (\Cref{ssec:stochastic-concentration}), obtaining a stochastic, high-probability, analog of our exact gradient result  (\Cref{ssec:stochastic-final}).

\subsection{Analysis in a ``good event''}\label{ssec:stochastic-good-event}

A careful inspection of our development thus far reveals that we only use the exact gradient assumption by substituting $\sum_{i < T} \inner{\gdiff[i](\eta)}{x_i(\eta)-\xopt} \ge 0$ into \Cref{lem:gm-error-bound-r}. Therefore, we consider the event where this inequality is approximately true. 
In particular, for $T\in \N$, and $\Gcoef,\Ccoef, \eta>0$ define
	\begin{equation*}
		\ev(\eta) \defeq  \crl*{
			\forall t\le T: \sum_{i < t} \inner{\gdiff[i](\eta)}{ x_i(\eta) - \xopt}
			\ge - \frac{1}{4}\max\{\dbar[t](\eta), \eta\sqrt{\Ccoef}\} \sqrt{\Gcoef \G[t](\eta) + \Ccoef}
		}.
	\end{equation*}

With this definition in hand, slightly modified versions of our key lemmas from
the deterministic analysis (Lemma~\ref{lem:distance-bound}, Proposition~\ref{prop:subopt-bound}, Lemma~\ref{lem:eta-max}) continue to hold. See~\Cref{app:stochastic:lem:distance-bound,app:prove-prop:stochastic:subopt-bound,app:prove-lem:stochastic-eta-max} for proofs of these results, which follow very similarly to their exact-gradient counterparts.

\begin{lem}\label{lem:stochastic-distance-bound}
	For any $T\in\N$, $\Gcoef > 2$ and $\Ccoef,\eta>0$, if event $\ev(\eta)$ holds and $\eta \le \bt(\eta)$ then
	$\dbar(\eta) \le \frac{3\Gcoef +2 }{\Gcoef -2} \d[0]
		\mbox{~~and~~}
		\rbar(\eta) \le \frac{4\Gcoef}{\Gcoef - 2} \d[0]$.
\end{lem}

\begin{prop}\label{prop:stochastic:subopt-bound}
	Let $T\in\N$, $\Gcoef > 2$, $\Ccoef > 0$, $\etalo > 0$ and $\etahi = 2^{2^k} \etalo$ for some $k \ge 1$. Let $\etaout = \rootfinding(\etalo,\etahi;T,\Gcoef,\Ccoef)$ and $\bar{x} = \frac{1}{T} \sum_{i < T} x_i(\etaout)$ where $\etalo \le \etahi$, and $\etahi > \bt(\etahi)$.
	If the event $\bigcap_{j=0}^{2^k} \ev( 2^j \etalo)$ holds and $\etalo \le  \bt(\etalo)$ then
	\begin{equation*}
		\norm{\bar{x} - x_0} \le \frac{4\Gcoef }{\Gcoef-2} \d[0]
		\mbox{~~and~~}
		f(\bar{x}) - f(\xopt) \le \frac{9\alpha - 2}{2(\alpha - 2)}  \cdot \frac{ \d[0] \sqrt{ \Gcoef \G(\eta') + \Ccoef}}{T}
	\end{equation*} 
	for some $\eta' \in [\eta, 2\eta]$. Moreover, if $\ev(\etalo)$ holds and  $\etalo > \bt(\etalo)$ then $\etaout = \etalo$ and
	\begin{equation*}
		\norm{\bar{x} - x_0} \le 
		\etalo \sqrt{\Gcoef \G(\etalo) + \Ccoef}
		\mbox{~~and~~}
		f(\bar{x}) - f(\xopt) \le \frac{5}{4} \frac{\d[0] \sqrt{\Gcoef \G(\etalo) + \Ccoef} + \etalo(\Gcoef\G(\etalo) + \Ccoef)}{T}.
	\end{equation*}
\end{prop}

\begin{lem}\label{lem:stochastic-eta-max}
	For any $T\in\N$, $\Gcoef > 2$ and $\Ccoef,\eta>0$, if event $\ev(\eta)$ holds then $\eta > \etamax(\Gcoef,\Ccoef) \defeq \frac{4\Gcoef}{\Gcoef-2}\cdot \frac{ \d[0]}{ \sqrt{\Gcoef \norm{g_0}^2 + \Ccoef}}$ implies that  
		$\eta > \bt(\eta)$.
	Consequently, if the event $\bigcap_{k=2,4,8,\ldots} \ev[T_k, \GcoefK, \CcoefK](2^{2^k}\etainit)$ holds, then 
	\Cref{alg:main} returns with $k \le 2 \log \log_+ \frac{\etamax(\GcoefK,\CcoefK)}{\etainit} \le  2 \log \log_+ \frac{\etamax(\GcoefK[0],\CcoefK[0])}{\etainit}$. 
\end{lem}

\subsection{The good event is likely}\label{ssec:stochastic-concentration}

We now arrive at the challenging part of the stochastic analysis: showing that the good event we defined occurs with high probability. For this, we require the following standard assumption.

\begin{assumption}\label{ass:lipschitz}
	The stochastic gradient oracle satisfies $\norm{\oracle(\eta)} \le L$ with probability 1.
\end{assumption}
\noindent
In online parameter-free optimization such assumption is unavoidable if one seeks regret scaling linearly in the comparator norm~\cite{cutkosky2017online}. However, similarly to the best prior results, our bounds depend on $L$ only via a low-order term. 

\newcommand{\themg}{X_i^{(k)}}

The following result shows that, for appropriate choices of $\Gcoef$ and $\Ccoef$ and any fixed $\eta \ge 0$ the event $\ev(\eta)$ has high probability.
\begin{prop}\label{prop:concentration-single-eta}
	Under \Cref{ass:lipschitz}, for any $T\in \N$, $\eta >0$ and $\delta\in(0,1)$, if
	$\Gcoef \ge 32^2 C
		~\mbox{and}~
		\Ccoef \ge (32CL)^2
		~\mbox{for}~
		C=  \log\prn*{\frac{60}{\delta} \log^2(6T)}$
	then $\Pr*( \ev(\eta ) ) \ge 1-\delta$.
\end{prop}
\Cref{prop:concentration-single-eta} makes no a-priori assumption on the size of $x_i(\eta)-\xopt$, instead controlling it empirically via $\dbar[t](\eta)$; this is unusual in the literature and crucial for our purposes. 
Our proof (given in \Cref{app:stochastic:signle-eta-proof}) relies on a time-uniform empirical-Bernstein-type martingale concentration bound~\cite{howard2021time}. However, since this result requires martingale differences that are bounded with probability 1, we cannot apply it on $\inner{\gdiff[i](\eta)}{x_i(\eta)-\xopt}$ (which is not bounded), nor can we apply it on $\inner{\gdiff[i](\eta)}{x_i(\eta)-\xopt}/\dbar[t](\eta)$ (which is bounded but is not adapted to any filtration). Instead, we consider processes of the form $\inner{\gdiff[i](\eta)}{\proj([x_i(\eta)-\xopt]/s)}$, where $\proj(\cdot)$ is the projection to the unit ball and $s$ is a fixed scalar. By carefully union bounding over a set of $O(\log T)$ values of $s$, we are able to control the probability of $\ev(\eta)$.

Having shown that the good event occurs with high probably for any fixed $\eta$, our next step is to show that, for proper choices of $\GcoefK$ and $\CcoefK$, good events hold with high probability for each and every single value of $\eta$ \Cref{alg:main} might try.
Noting that for, each value of $k$, \Cref{alg:main} only tests step size values of the form $2^j \etainit$ for $j\in \{0, \ldots, 2^k\}$, the following lemma (which is a direct application of union bounds) provides the required guarantee; see proof in \Cref{app:stochastic-concentration-master-proof}.

\begin{lem}\label{lem:concentration-master}
	For budget $B\in \N$, initial step size $\etainit > 0$, and failure probability $\delta\in(0,1)$, let
	\begin{equation}\label{eq:coef-k-setting}
		\GcoefK \defeq 32^2 C_k
		~~\mbox{and}~~
		\CcoefK \defeq (32C_k L)^2,~~\mbox{where}~~
		C_k = {2k + \log \prn*{\frac{60 \log^2(6B)}{\delta}}}.
	\end{equation}
	Then, under \Cref{ass:lipschitz}, we have
	$\Pr*(  \bigcap_{k=2, 4, 8, \ldots} \bigcap_{j=0, 1, \ldots, 2^k} \ev[B, \GcoefK, \CcoefK]( 2^j \etainit ) ) \ge 1-\delta$.
\end{lem}

\subsection{Parameter-free rates for stochastic convex optimization}\label{ssec:stochastic-final}

We are ready to state our main result; see proof in \Cref{app:stochastic:main-proof}.
\begin{theorem}\label{thm:stochastic-main}
	Under \Cref{ass:lipschitz}, for any $\delta\in (0,1)$ consider \Cref{alg:main}  with parameters $\GcoefK$, $\CcoefK$ given by \cref{eq:coef-k-setting}, $\etainit > 0$, $B\in \N$, and $x_0\in\R^d$. The algorithm makes at most $B$ gradient queries and returns $\bar{x} = \frac{1}{T}\sum_{i< T}x_i(\eta)\in \R^d$ for  $\eta \ge \etainit$ such that, with probability at least $1-\delta$,
	\begin{equation}\label{eq:stochastic-T-lb}
		T \ge \max\crl*{ \frac{B}{8 \log \log_+ \frac{\norm{x_0 - \xopt}}{\etainit L}}, 1 }
	\end{equation}
	and for $C = \log \frac{1}{\delta} +  \log\log_+  \frac{B\norm{\xopt-x_0}}{\etainit L}$, either 
	\begin{equation}\label{eq:stochastic-main-bound-normal-case}
		\norm{\bar{x}-\xopt} \le 6\norm{x_0 - \xopt}
		~~\mbox{and}~~
		f(\bar{x}) - f(\xopt) \le O\prn*{ \frac{ \norm{x_0 - \xopt}\sqrt{C\G(\eta')+C^2L^2}}{T}}
	\end{equation}
	for some $\eta' \in [\eta, 2\eta]$, or
	\begin{equation}\label{eq:stochastic-main-bound-edge-case}
		\norm{\bar{x}-\xopt} = O\prn*{ \etainit \sqrt{C\G(\etainit)+C^2L^2}}
		~~\mbox{and}~~
		f(\bar{x}) - f(\xopt) \le O\prn*{\frac{\etainit \prn*{C\G(\etainit)+C^2L^2}}{T}}.
	\end{equation}
\end{theorem}

Let us compare our bounds to the best known prior bounds, which follow from  from online to batch conversion of parameter-free regret bounds.
\citet{mcmahan2014unconstrained} achieve an optimal parameter-free regret bound for algorithms that are not adaptive to gradient norms: 
For any user-specified $\varepsilon$ and gradient budget $B$, their result guarantees an \emph{expected} optimality gap of $O\prn[\big]{\frac{\varepsilon + {\d[0] L}\sqrt{\Lambda}}{\sqrt{B}}}$ where $\Lambda = \log \prn[\big]{ 1+ \frac{\d[0] L}{ \varepsilon} }$ is logarithmic in $\frac{1}{\varepsilon}$. In comparison, by taking $\etainit = O(\frac{\varepsilon}{L^2 B})$ we guarantee a probability $1-\delta$ optimality gap of $O\prn[\big]{\frac{(\varepsilon + {\d[0] L})\lambda^2 }{\sqrt{B}}}$,
where $\lambda = \log \prn*{\frac{1}{\delta} \log_+ \frac{B \d[0] L}{\varepsilon}}$ is only double-logarithmic in $\frac{1}{\varepsilon}$; see \Cref{app:stochastic:main-coro-proof} for a slightly tighter bound in this  setting.

\citet{cutkosky2018black} provide the best known parameter-free regret bound for algorithms adaptive to gradient norms. Letting $G=\sum_{i<B} \norm{\oracle(x_i)}^2$ denote the sum of all squared gradient norms observed during optimization, they guarantee expected optimality gap  $O\prn[\big]{\frac{\varepsilon + {\d[0]}\E \sqrt{\Lambda G + {\Lambda}^2 L^2}}{B}}$ for $\Lambda = \log \prn[\big]{ 1+ \frac{\d[0] \sqrt{G}}{ \varepsilon}}$. In comparison, 
substituting the somewhat crude bound $\G(\eta') \le G$, we guarantee a probability $1-\delta$ optimality gap of $O\prn[\big]{\frac{\varepsilon \lambda^2 + {\d[0] \sqrt{\lambda^2 G  +\lambda^3 L^2 }} }{B}}$ for the double-logarithmic $\lambda$ defined above.

For small values of $\varepsilon$, our bounds show a clear asymptotic improvement over prior art. However, we note that for a hypothetical optimally-tuned $\varepsilon$ (which depends on the unknown problem parameter $\d[0]$), the logarithmic factor $\Lambda$ of prior work becomes $O(1)$, while our double-logarithmic factor $\lambda$ remains $O(\log(\frac{1}{\delta}\log B))$. This occurs because Lemma~\ref{lem:stochastic-eta-max} only provides a somewhat loose bound on $\eta_{\max}$, and because of the union bound in the proof of \Cref{prop:concentration-single-eta}. We can mitigate this issue at a cost of adaptivity to gradient norm; see \Cref{sec:alternative-bisection} for further discussion.

Our results give the the first high-probability parameter-free rates. However,
while high-probability bounds are generally considered stronger than expectation bounds, it is not clear how to deduce an expectation bound from our results without increasing the error by a $\mathsf{poly}(\log B)$ factor, due to the need to set $\delta = \mathsf{poly}(1/B)$. 
Finally, we note that \Cref{thm:stochastic-main} also guarantees that the output of the algorithm is at most a multiplicative factor further from $\xopt$ than $x_0$ was; we believe that this type of guarantee is new in the parameter-free setting.\footnote{\citet[Lemma 25]{orabona2021parameter} bound the distance moved by Follow the Regularized Leader iterates, but not by a multiple of $\norm{\xopt-x_0}$.}

\section{Adaptivity to problem structure}\label{sec:adaptivity}

\arxiv{
In this section we showcase our algorithm's adaptivity by proving stronger rates of convergence under smoothness and strong-convexity assumptions, without introducing any new parameters.
}

\subsection{Adaptivity to smoothness with exact gradients}\label{sec:adaptivity-smooth}

Let us assume that $f$ is $S$-smooth (i.e., has $S$-Lipschitz gradient), and consider for simplicity the exact gradient setting; we believe that similar results extend to the stochastic setting as well. Under these assumptions, we show that \Cref{alg:main}, \emph{without any changes}, achieves (up to double-logarithmic factors) the  $S\d[0]^2 /T$ suboptimality bound of optimally-tuned GD, as long as $\etainit < \frac{1}{2S}$. See \Cref{app:proof-of:thm:adaptivity-to-smoothness} for proof.

\begin{theorem}\label{thm:adaptivity-to-smoothness}
Consider the noiseless regime and assume $f$ is $\sm$-smooth. Then \Cref{alg:main}, with gradient budget $B\in\N$, parameters $\GcoefK=3$, $\CcoefK=0$ (as in \Cref{thm:deterministic-main}), and $\etainit \le \frac{1}{2 \sm}$,  returns $\bar{x}\in\R^d$ that satisfies%
\notarxiv{$f(\bar{x}) - f(\xopt) \le  O\prn*{ \frac{S\norm{x_0 - \xopt}^2 }{B}   \log \log_+ \frac{\norm{\xopt - x_0}}{\etainit \norm{g_0}}  }$.}%
\arxiv{\begin{equation*}
	f(\bar{x}) - f(\xopt) \le  O\prn*{ \frac{S\norm{x_0 - \xopt}^2 }{B}   \log \log_+ \frac{\norm{\xopt - x_0}}{\etainit \norm{g_0}}  }.
\end{equation*}}
\end{theorem}

We remark that UniXGrad~\cite{kavis2019unixgrad} features optimal (accelerated) rates for smooth problems without dependence on the parameter $S$. However, unlike our method, it requires knowledge of $\d[0]$.

\subsection{Adaptivity strong convexity using restarts}\label{sec:adaptivity-sc}
\newcommand{\pind}[1]{^{(#1)}}
\newcommand{\pftuner}{\textsc{ParameterFreeTuner}}

We now consider a standard strongly-convex stochastic setup~\cite[e.g.,][]{hazan2014beyond,shamir2013stochastic} in which we assume $f$ to be $\mu$-strongly-convex in $\xset$ and admit a stochastic gradient oracle bounded by $L$. (Note that this implies a bound of $L/\mu$ on the diameter of $\xset$). \citet{hazan2014beyond} propose to run SGD for epochs of  doubling length and halving step sizes. For a total gradient budget of $B$, they obtain the optimal bound $O(L^2/(\mu B))$ on the expected optimality gap. However, their scheme requires the initial step size to be proportional to $1/\mu$, and hence requires knowledge of $\mu$. 

We show that restarting \Cref{alg:main} with doubling gradient budgets (and no step size to tune) recovers (up to double-logarithmic factors) the optimal $1/B$ rate of convergence. To describe the procedure formally, let $\pftuner(x_0, B, \delta, \etainit)$ denote the output of \Cref{alg:main} with initial point $x_0$, gradient budget $B$, failure probability $\delta$, minimal step size $\etainit$ and $\GcoefK, \CcoefK$ as in~\cref{eq:coef-k-setting}. For user-specified $\varepsilon>0$ and $\delta\in(0,1)$ and $x\pind{0}=x_0$, our doubling procedure is 
\begin{equation*}
	x\pind{m} \gets \pftuner\prn*{ x\pind{m-1}, B\pind{m} \defeq 2^m, \delta\pind{m} \defeq \tfrac{1}{m(m+1)} \delta, \etainit\pind{m} \defeq \frac{\varepsilon}{L^2 B\pind{m} }}.
\end{equation*}

\newcommand{\levelset}{\mathcal{L}}
\renewcommand{\sb}{b}
\newcommand{\Bscv}{B}
\newcommand{\Cscv}{C}

\begin{theorem}\label{thm:strong-convexity-result}
For any $\varepsilon > 0$, $\delta\in(0,1)$, and $M\in\N$, computing $x\pind{M}$ in the procedure described above requires at most $B=2^{M+1}$ gradient queries. 
If $f$ is $\mu$-strongly convex and has a stochastic gradient oracle bounded by $L$, then with probability at least $1-\delta$ we have
\notarxiv{$f(x^{(M)}) - f(x_\star)  \le O\prn*{
 \frac{L^2 / \mu + \varepsilon}{B} \log^3 \prn*{\frac{1}{\delta}\log_+ \frac{B \norm{x_0-\xopt}L}{\varepsilon}}}$.}
\arxiv{
\begin{equation*}
	f(x^{(M)}) - f(x_\star)  \le O\prn*{
		\frac{L^2 / \mu + \varepsilon}{B} \log^3 \prn*{\frac{1}{\delta}\log_+ \frac{B \norm{x_0-\xopt}L}{\varepsilon}}}.
\end{equation*}
}
\end{theorem}
\noindent
See proof in \Cref{app:proof-of:thm:strong-convexity-result}. Compared to results obtained via parameter-free strongly-convex regret bounds~\cite[Thm.~7]{cutkosky2018black}, we remove a squared logarithmic factor, breaking two  regret minimization barriers at once.

\arxiv{\newcommand{\acks}[1]{\section*{Acknowledgment} #1}}
\acks{
We thank Shira Baneth and Mikey Shechter for pointing out typos in an earlier version of this paper.
YC was partially supported by 
the Len Blavatnik and the Blavatnik Family Foundation and an Alon Fellowship.
OH was partially supported by the Pitt Momentum Funds.
}

\newpage

\arxiv{\bibliographystyle{abbrvnat}}

\newpage

\appendix

\section{Proofs for Section~\ref{sec:deterministic}}\label{app:proof-of:deterministic}

\subsection{Proof of Lemma~\ref{lem:output-properties}}\label{app:proof-of:lem:output-properties}

\begin{proof} 
	Consider the case when $\rootfinding$ returns $\etaout=\etalofinal$, i.e., when 
	\begin{equation*}
		\frac{\rbar(\etalofinal)}{2\rbar(\etahifinal)}\bt(\etahifinal) \overle{(i)} 
		\half \etahifinal 
		\overle{(ii)} 
		\etalofinal
		\overle{(iii)}
		\bt(\etalofinal)
	\end{equation*}
	due to $(i)$ the condition for returning $\etaout=\etalofinal$, $(ii)$ the bisection termination condition $\etahifinal \le 2\etalofinal$, and $(iii)$ the bisection invariant. Substituting $\bt(\etaout) = \rbar(\etaout)/\sqrt{\Gcoef \G(\etaout) + \Ccoef}$
	into the above display yields
	\begin{equation}\label{eq:eta-sandwitch-lo}
		\frac{\rbar(\etalofinal)}{2\sqrt{\Gcoef \G(\etahifinal) + \Ccoef}}
		\le 
		\etalofinal
		\le 
		\frac{\rbar(\etalofinal)}{\sqrt{\Gcoef \G(\etalofinal) + \Ccoef}}.
	\end{equation}
	Next, consider the case when $\rootfinding$ returns $\etaout=\etahifinal$, i.e., when 
	\begin{equation}\label{eq:eta-sandwitch-hi}
		\frac{\rbar(\etahifinal)}{\sqrt{\Gcoef \G(\etahifinal) + \Ccoef}}
		=
		\bt(\etahifinal) \overle{(i)} \etahifinal \overle{(ii)} \frac{\rbar(\etalofinal)}{\rbar(\etahifinal)}\bt(\etahifinal)
		=
		\frac{\rbar(\etalofinal)}{\sqrt{\Gcoef \G(\etahifinal) + \Ccoef}},
	\end{equation}
	due to $(i)$ the bisection invariant, and $(ii)$  the condition for returning $\etaout=\etahifinal$. 
	
	Finally, $\rbar(\etaout) \le \rbar(\etalofinal)$ holds trivially if $\etaout = \etalofinal$, and if $\etaout = \etahifinal$ then it holds
	by \eqref{eq:eta-sandwitch-hi}; similarly $\sqrt{\Gcoef \G(\etaout) + \Ccoef} \le 2  \sqrt{\Gcoef \G(\etahifinal) + \Ccoef}$ holds either trivially or from~\eqref{eq:eta-sandwitch-lo}.
\end{proof}

\subsection{Proof of Proposition~\ref{prop:subopt-bound}}\label{app:proof-of:subopt-bound}

\begin{proof}
We begin with the first case of the proposition, assuming that the initial $\etalo,\etahi$  satisfy $\etalo \le  \bt(\etalo)$ and $\etahi > \bt(\etahi)$ so that $\rootfinding$ terminates in  \cref{line:bisection-normal-termination} with final interval $[\etalofinal,\etahifinal]$. In this case, we have the following error bound
\begin{flalign*}
	f(\bar{x}) - f(x_\star) & \overle{(i)} \frac{\rbar(\etaout)\d[0]}{\etaout T} + \frac{\etaout \G[T](\etaout)}{2T} \numberthis\label{eq:gm-error-bound-r-restate}\\
	&\overle{(ii)} \frac{ 2 \d[0]\sqrt{\Gcoef \G(\etahifinal) + \Ccoef} }{T} +  \frac{ \rbar(\etalofinal)}{2 \Gcoef T} \sqrt{\Gcoef \G[T](\etaout) + \Ccoef} \\
	&\overle{(iii)} \frac{ 2 \d[0]\sqrt{\Gcoef \G(\etahifinal) + \Ccoef} }{T} +  \frac{\d[0]}{T (\Gcoef - 1)} \sqrt{\Gcoef \G[T](\etaout) + \Ccoef} \\
	&\overle{(iv)} \frac{ 2 \d[0]\sqrt{\Gcoef \G(\etahifinal) + \Ccoef} }{T} + \frac{2 \d[0]}{T (\Gcoef - 1)}  \sqrt{\Gcoef \G[T](\etahifinal) + \Ccoef} \\
	&\overeq{\phantom{(iv)}} \frac{2 \alpha}{\alpha - 1} \cdot \frac{\d[0] \sqrt{\Gcoef \G[T](\etahifinal) + \Ccoef} }{T},
\end{flalign*}
due to $(i)$ \Cref{lem:gm-error-bound-r}, 
$(ii)$ \Cref{lem:output-properties}, 
$(iii)$ \Cref{lem:distance-bound} and 
$(iv)$ \Cref{lem:output-properties} again.	 Moreover, we have $\norm{\bar{x}-x_0} \le \rbar(\etaout) \le \rbar(\etalofinal) \le  \frac{2\Gcoef}{\Gcoef-1}\d[0]$ from \Cref{lem:output-properties,lem:distance-bound}.

Next, consider that case where $\rootfinding$ returns the initial $\etalo$, i.e., where $\etalo > \bt(\etalo) = \rbar(\etalo) / \sqrt{\Gcoef \G(\etalo) + \Ccoef}$. Rearranging this condition immediately yields the claimed bound on $\norm{\bar{x}(\etalo) - x_0} \le \rbar(\etalo)$. Substituting the lower bound on $\etaout$ into the error bound~\eqref{eq:gm-error-bound-r-restate} completes the proof.
\end{proof}

\subsection{Proof of Lemma~\ref{lem:eta-max}}\label{app:proof-of:lem:eta-max}

\begin{proof}
	Assume by contradiction that $\eta > \etamax$ but $\eta \le \bt(\eta)$. On the one hand, 
	\begin{equation*}
		\frac{\rbar(\eta)}{\sqrt{\Gcoef \norm{g_0}^2 + \Ccoef}} \ge 
		\frac{\rbar(\eta)}{\sqrt{\Gcoef \G(\eta) + \Ccoef}} = \bt(\eta) \ge \eta > \etamax =  \frac{2\Gcoef}{\Gcoef-1}\cdot \frac{ \d[0]}{ \sqrt{\Gcoef \norm{g_0}^2 + \Ccoef}}.
	\end{equation*}
	which implies $\rbar(\eta) > \frac{2\Gcoef}{\Gcoef -1}\d[0]$.
	On the other hand, by \Cref{lem:distance-bound}, we have $\rbar(\eta) \le \frac{2\Gcoef}{\Gcoef -1}\d[0]$ which yields a contradiction.
	
	By the discussion so far, the doubling scheme in \Cref{alg:main} must stop when $2^{2^k}\etainit > \etamax$. 
	Therefore, if it stops with $k=\kout$ we may conclude that either $\kout = 2$ or $2^{2^{\kout/2}}\etainit \le \etamax$, as otherwise it would have stopped with $\kout/2$ instead. Rearranging yields $\kout \le 2 \log \log_+ \frac{\etamax}{\etainit}$ as claimed.
\end{proof}

\subsection{Proof of Theorem~\ref{thm:deterministic-main}}\label{app:proof-of:thm:deterministic-main}

\begin{proof}
	Let us first show that \Cref{alg:main} never makes more than $B$ gradient queries. The algorithm repeatedly calls $\rootfinding$, with parameters $T=T_k = \floor{\frac{B}{2k}}\le\frac{B}{2k}$ and $\etahi/\etalo = 2^{2^k}$ for $k=2,4,8,\ldots$, until it returns $\etaout < \infty$ (or $k$ passes $B/8$). In the iterations where $\rootfinding$ returns $\etaout = \infty$ it uses only $T_k$ subgradient queries (a single evaluation of $\bt$). In the final iteration where $\rootfinding$ returns $\etaout < \infty$, it evaluates $\bt$ at most $k+2$ times: once at \cref{line:bisection-etahi-check}, once at \cref{line:bisection-etalo-check}, and $\log\log(\etahi/\etalo) = k$ times at \cref{line:bisection-etamid-check} during the bisection. Therefore, letting $k'=2^{j'}$ denote the index of the final iteration, the total query complexity of the final bisection call is 
	\begin{equation*}
		(k'+2) T_{k'}\le (2+k') \frac{B}{2k'} \le \frac{B}{2} + \frac{B}{2^{j'}},
	\end{equation*}
	 and the complexity of all preceding bisection calls is at most \begin{equation*}
	 	\sum_{j=1}^{j'-1} T_{2^j} \le \frac{B}{2}\sum_{j=1}^{j'-1} 2^{-j} = \frac{B}{2} - \frac{B}{2^{j'}},
	 \end{equation*}
	 giving a total complexity bound of $B$.

	Next, we establish \eqref{eq:determinisitic-T-lb}. Note that \Cref{alg:main} indeed always returns a point of the form $\bar{x} = \frac{1}{T}\sum_{i< T}x_i(\eta)$ with some $T\ge 1$; the edge case of returning in \cref{line:main-alg-edge-case-termination} corresponds to $T=1$. Moreover, in the typical case of returning in \cref{line:main-alg-normal-termination}, we have $T=\floor*{\frac{B}{2k}} \ge \frac{B}{4k}$ for some $k\le B/4$. By \Cref{lem:eta-max}, we have 
	\begin{equation*}
		k \le 2 \log \log_+ \frac{\etamax}{\etainit} = 2 \log \log_+ \frac{\sqrt{3} \d[0]}{\etainit  \norm{g_0}} \le 3\log \log_+ \frac{\d[0]}{\etainit  \norm{g_0}}, 
	\end{equation*}
	giving the claimed lower bound~\eqref{eq:determinisitic-T-lb} on $T$. 
	
	It remains to show that one of the  conclusions~\eqref{eq:deterministic-main-bound-normal-case} and~\eqref{eq:deterministic-main-bound-edge-case} must hold. When \Cref{alg:main} returns in \cref{line:main-alg-normal-termination}, this  follows immediately from~\Cref{prop:subopt-bound} (if 
	$\etainit \le \bt(\etainit)$ then conclusion~\eqref{eq:deterministic-main-bound-normal-case} holds; if  
	$\etainit > \bt(\etainit)$ 
	then either 
	$\etainit \sqrt{3 \G(\etainit)} \le 3\d[0]$
	and
	conclusion~\eqref{eq:deterministic-main-bound-normal-case} holds, or  $\d[0]\sqrt{3 \G(\etainit)} \le \etainit \G(\etainit)$ and conclusion~\eqref{eq:deterministic-main-bound-edge-case} holds).
	In the edge case of returning $\bar{x}=x_0$ in \cref{line:main-alg-edge-case-termination}, corresponding to $T=1$, conclusion~\eqref{eq:deterministic-main-bound-normal-case} clearly holds, as $\norm{x_0-\xopt}\le 4\norm{x_0-\xopt}$ trivially and $f(x_0)-f(\xopt) \le \inner{g_0}{x_0-\xopt} \le \norm{x_0-\xopt}\norm{g_0}$ due to convexity of $f$. We remark that due to inequality~\eqref{eq:determinisitic-T-lb},  the $T=1$ edge case is only possible for a very small iteration budget $B = O(\log\log_+ \frac{\norm{x_0-\xopt}}{\etainit\norm{g_0}})$.
\end{proof}

\section{Proofs for \Cref{sec:stochastic}}

\subsection{Proof of \Cref{lem:stochastic-distance-bound}}\label{app:stochastic:lem:distance-bound}

\begin{proof}
	The proof proceeds similarly to the proof of \Cref{lem:distance-bound}, except that instead of assuming exact subgradients we make use of the definition of $\ev$.
	As usual in proofs where $\eta$ is fixed, we drop the explicit dependence on it from  $x_t, g_t, \d[t], \dbar[t], \r[t], \rbar[t]$ and $\G[t]$.

	We start with the following inequality, which holds for any $t\in\N$ by summing Eq.~\eqref{eq:standard-sgm-inequality} and recalling that $\gdiff[i] = g_i - \grad f(x_i)$,
	\begin{equation*}
		\d[t]^2 
				\le \d[0]^2 + \eta^2 \G[t] 
				- 2\eta \sum_{i< t} \inner{g_i}{x_i-\xopt}
		= \d[0]^2 + \eta^2 \G[t] - 2\eta \sum_{i< t} \inner{\grad f(x_i)}{x_i-\xopt} -2 \eta \sum_{i< t} \inner{\gdiff[i]}{x_i-\xopt}.
	\end{equation*}
	Noting that $\inner{\grad f(x_i)}{x_i - \xopt} \ge f(x_i)  - f(\xopt) \ge 0$ due to convexity and that $\sum_{i< t} \inner{\gdiff[i]}{x_i-\xopt} \ge -\frac{1}{4}\max\{\dbar[t],\eta\sqrt{\Ccoef}\}\sqrt{\Gcoef \G[t] + \Ccoef}$ for all $t\le T$ due to $\ev$ holding, we have
	\begin{equation*}
		\d[t]^2 \le \d[0]^2 + \eta^2 \G[t] 
		+ \frac{1}{2} \max\{\eta\dbar[t],\eta^2\sqrt{\Ccoef}\}\sqrt{\Gcoef \G[t] + \Ccoef}
	\end{equation*}
	for all $t \le T$. 
	
	Maximizing both sides of the inequality over $t\le T$ and recalling that $\eta \le \bt(\eta) = \rbar / \sqrt{\Gcoef \G + \Ccoef}$, we get
	\begin{equation*}
		\dbar^2 \le \d[0]^2 + \frac{1}{\Gcoef}\rbar^2 +  \max\crl*{ \frac{\rbar \dbar}{2}, \frac{\rbar^2}{2\sqrt{\Gcoef }} }.
	\end{equation*}
	Substituting $\rbar \le \dbar + \d[0]$ (which holds due to the triangle inequality), we get
	\begin{equation*}
		\dbar^2 \le \d[0]^2 + \frac{1}{\Gcoef} (\dbar + \d[0])^2 +  \max\crl*{ \frac{(\dbar + \d[0])^2}{2}, \frac{(\dbar + \d[0])^2}{2\sqrt{\Gcoef}} }
		= \d[0]^2 + \underset{1/\alpha'}{\underbrace{\prn*{\frac{1}{\Gcoef} + \frac{1}{2}}}} (\dbar + \d[0])^2,
	\end{equation*} 
	where the final equality is due to $\Gcoef > 1$. Thus, we arrive again at inequality~\eqref{eq:deterministic-quadratic-ineq} from the proof of \Cref{lem:distance-bound}, but with $\Gcoef$ replaced by $\Gcoef' = 2\Gcoef/(\Gcoef+2)$. We consequently find that
	\begin{equation*}
		\dbar \le \frac{\Gcoef' + 1}{\Gcoef' - 1}\d[0] = 
		\frac{3\Gcoef + 2}{\Gcoef - 2}\d[0]
		~~\mbox{and}~~
		\rbar \le \dbar + \d[0] \le \frac{4\Gcoef}{\Gcoef - 2}\d[0].
	\end{equation*}
\end{proof}

\subsection{Proof of \Cref{prop:stochastic:subopt-bound}}\label{app:prove-prop:stochastic:subopt-bound}

\begin{proof}		
	In both cases of the proposition, 
	the event $\ev(\etaout)$ holds
	(in the first case because $2^j \etalo$ for $j = 1, \dots, 2^k$ represents all possible values for $\etaout$; in the second case because $\etaout = \etalo$ by Line~\ref{line:bisection-etalo-check} of $\rootfinding$) and therefore,
	\begin{flalign}\label{eq:bound-delta-i-sum}
		\sum_{i < T} \inner{\gdiff[i](\etaout)}{ x_\star  - x_i(\etaout)} \le \frac{1}{4} \max\left\{ \dbar(\etaout) , \etaout \sqrt{\beta} \right\} \sqrt{\Gcoef \G[T](\etaout) + \Ccoef}.
	\end{flalign}

	We begin with the first case of the lemma (assuming $\etalo \le \etahi$ satisfy $\etalo \le  \bt(\etalo)$ and $\etahi > \bt(\etahi)$).
	In this case, \rootfinding{} terminates at \cref{line:bisection-normal-termination}.
	Let $\etalofinal$ and $\etahifinal$ denote the values of $\etalo$ and $\etahi$ at \cref{line:bisection-normal-termination} of $\rootfinding$, respectively.
	First, note that, by \eqref{eq:bound-delta-i-sum} and \Cref{lem:output-properties} we have
	\begin{flalign}
		\sum_{i < T} \inner{\gdiff[i](\etaout)}{ x_\star  - x_i(\etaout)} &\le  \frac{1}{4}
		\max\left\{ \dbar(\etaout) , \rbar(\etalofinal) \right\} \sqrt{\Gcoef \G[T](\etaout) + \Ccoef} \notag \\
		&\le \frac{\d[0] + \rbar[T](\etalofinal)}{4} \sqrt{\Gcoef \G[T](\etaout) + \Ccoef} .
		\label{eq:delta-i-sum-upperbound-in-main-case}
	\end{flalign}
	where the first inequality uses \eqref{eq:bound-delta-i-sum}, $\etaout \le \frac{\rbar(\etalofinal)}{\sqrt{\Gcoef \G(\etaout) + \Ccoef}}$ (\Cref{lem:output-properties}) and $\sqrt{\beta} \le \sqrt{\Gcoef \G(\etaout) + \Ccoef}$, the second inequality is by $\dbar(\etaout) \le \d[0] + \rbar[T](\etaout)$ and $\rbar(\etaout) \le \rbar(\etalofinal)$ (\Cref{lem:output-properties}).
	
	The remainder of the proof is very similar to the proof of \Cref{prop:subopt-bound}.
	Combining \eqref{eq:gm-error-bound-r}, \eqref{eq:output-property} and \eqref{eq:delta-i-sum-upperbound-in-main-case}  yields:
	\begin{flalign*}
		f(\bar{x}) - f(x_\star) %
		&\le \frac{ 2 \d[0]\sqrt{\Gcoef \G(\etahifinal) + \Ccoef} }{T} +  \frac{2 \alpha^{-1} \rbar(\etalofinal) + \d[0] + \rbar[T](\etalofinal)}{4 T} \sqrt{\Gcoef \G[T](\etaout) + \Ccoef} \\
		&\le \frac{ 2 \d[0]\sqrt{\Gcoef \G(\etahifinal) + \Ccoef} }{T} +  \frac{ (2 \alpha^{-1} + 1) \frac{4 \alpha}{\alpha - 2} + 1}{4 T} \d[0] \sqrt{\Gcoef \G[T](\etaout) + \Ccoef} \\
		&\le \frac{ 2 \d[0]\sqrt{\Gcoef \G(\etahifinal) + \Ccoef} }{T} +  \frac{ (2 \alpha^{-1} + 1) \frac{4 \alpha}{\alpha - 2} + 1}{2 T} \d[0] \sqrt{\Gcoef \G[T](\etahifinal) + \Ccoef} \\
		&= \frac{9\alpha - 2}{2(\alpha - 2)} \cdot \frac{\d[0] \sqrt{\Gcoef \G[T](\etahifinal) + \Ccoef} }{T} 
	\end{flalign*}
	where the first inequality substitutes \eqref{eq:output-property} and \eqref{eq:delta-i-sum-upperbound-in-main-case} into
	\eqref{eq:gm-error-bound-r}, the second inequality uses that $\rbar[T](\etalofinal) \le \frac{4 \alpha}{\alpha - 2} \d[0]$ by $\etalofinal \le \phi(\etalofinal)$ and \Cref{lem:stochastic-distance-bound}, the third inequality uses  
	$\sqrt{\Gcoef \G[T](\etaout) + \Ccoef} \le 2 \sqrt{\Gcoef \G[T](\etahifinal) + \Ccoef}$ by \Cref{lem:output-properties}, and the
	final equality is algebra.
	
	Finally, consider that case where $\rootfinding$ returns the initial $\etalo$, i.e., where $\etalo > \bt(\etalo) = \rbar(\etalo) / \sqrt{\Gcoef \G(\etalo) + \Ccoef}$. Rearranging this condition immediately yields the claimed bound on $\norm{\bar{x}(\etalo) - x_0} \le \rbar(\etalo)$. Substituting the lower bound on $\etalo$ into the error bound~\eqref{eq:gm-error-bound-r} and applying \eqref{eq:bound-delta-i-sum} using $\dbar(\etalo) \le \rbar(\etalo) + \d[0]$ completes the proof.
\end{proof}

\subsection{Proof of \Cref{lem:stochastic-eta-max}}\label{app:prove-lem:stochastic-eta-max}

\begin{proof}
	The proof is essentially identical to the proof of \Cref{lem:eta-max}, with \Cref{lem:stochastic-distance-bound} used instead of \Cref{lem:distance-bound}; we give it here for completeness.
	
	To show the first part of the lemma, assume that $\ev(\eta)$ holds and assume by contradiction that $\eta > \etamax(\Gcoef,\Ccoef)$ but $\eta \le \bt(\eta)$. On the one hand, 
	\begin{equation*}
		\frac{\rbar(\eta)}{\sqrt{\Gcoef \norm{g_0}^2 + \Ccoef}} \ge 
		\frac{\rbar(\eta)}{\sqrt{\Gcoef \G(\eta) + \Ccoef}} = \bt(\eta) \ge \eta > \etamax(\Gcoef,\Ccoef) =  \frac{4\Gcoef}{\Gcoef-2}\cdot \frac{ \d[0]}{ \sqrt{\Gcoef \norm{g_0}^2 + \Ccoef}}.
	\end{equation*}
	which implies $\rbar(\eta) > \frac{4\Gcoef}{\Gcoef -2}\d[0]$.
	On the other hand, by \Cref{lem:stochastic-distance-bound}, we have $\rbar(\eta) \le \frac{4\Gcoef}{\Gcoef -2}\d[0]$ which yields a contradiction.
	
	\Cref{alg:main} repeatedly invokes $\rootfinding$ with $\etahi$ values of the form $2^{2^k} \etainit$, $\Gcoef=\GcoefK$, and $\Ccoef=\CcoefK$ (for $k=2,4,\ldots$) until the bisection returns $\etaout < \infty$. This happens as soon as $\etahi >  \bt(\etahi)$, which holds (by the discussion above and since we assume $\ev[T_k, \CcoefK, \CcoefK](2^{2^k}\etainit)$ holds for all $k=2,4,\ldots$)  whenever $\etahi > \etamax(\GcoefK,\CcoefK)$.
	Therefore, if the algorithm returns with $k=\kout$ we may conclude that either $\kout = 2$ or $2^{2^{\kout/2}}\etainit \le \etamax(\GcoefK[\kout],\CcoefK[\kout])$, as otherwise we would have returned with $\kout/2$ instead. Rearranging yields $\kout \le 2 \log \log_+ \frac{\etamax(\GcoefK[\kout],\CcoefK[\kout])}{\etainit}$ as claimed. Finally, we note that $\etamax(\Gcoef,\Ccoef)$ is decreasing in both $\Gcoef$ and $\Ccoef$, and we may therefore replace $\GcoefK,\CcoefK$ by the larger values $\GcoefK[0],\CcoefK[0]$ as defined in eq.~\eqref{eq:coef-k-setting}.
\end{proof}

\subsection{A martingale concentration bound}

The following corollary is a translation of \cite[][Theorem 4]{howard2021time} which simplifies notation at the cost of looser constants. We remark that it  holds even when $\log$ denotes the natural logarithm (as is the convention in~\cite{howard2021time}).

\begin{corollary}[of {\citet[][Theorem 4]{howard2021time}}]\label{cor:kickass-mg-concentration}
Let $X_{t}$ be adapted to $\filt_{t}$ such that $\left|X_{t}\right|\le1$
with probability 1 for all $t$. Then, for every $\delta\in\left(0,1\right)$
and any $\hat{X}_{t}\in\filt_{t-1}$ such that $\abs{\hat{X}_{t}}\le1$
with probability 1,
\begin{equation*}
	\P\prn*{
		\exists t<\infty:\abs[\Bigg]{\sum_{s=1}^t \prn*{X_{s}-\Ex{X_{s}|\filt_{s-1}}}
		}
		\ge
		4 \sqrt{ A_t(\delta) \sum_{s= 1}^t \left(X_{s}-\hat{X}_{s}\right)^{2} + A_t^2(\delta)}
		\,}
	\le \delta,
\end{equation*}
where $A_t(\delta) = \log \prn*{
	\frac{60 \log(6t)}{\delta}}$. %
\end{corollary}

\begin{proof}
	Throughout we proof we use the binary maximization notation $a\vee b \defeq \max\{a,b\}$.
	
	We apply \cite[][Theorem 4]{howard2021time} with $a=-b=1$ and the polynomial stitched boundary \cite[][Eq.~(10)]{howard2021time} with parameters $m,\eta,s\ge 1$ to be specified below. This yields
	\begin{equation*}
		\Pr*(
			\exists t < \infty : \abs[\Bigg]{\sum_{s\le t}\prn*{X_{s}-\Ex{X_{s}|\filt_{s-1}}}
			}
			\ge
			S_{\delta/2}\prn*{ m \vee \sum_{s\le t}\prn*{X_s -\hat{X}_s}^2 }
		) \le \delta,
	\end{equation*}
	where, for every $v'>0$
	\begin{equation*}
		S_{\delta/2}\left(v'\right):=k_{1}\sqrt{ v'\left(s\log\log\left(\frac{\eta v'}{m}\right)
			+\log\frac{2\zeta\left(s\right)}{\delta\log^{s}\eta}
			\right)
			}
		+2k_{2}\left(s\log\log\left(\frac{\eta v'}{m}\right)+\log\frac{2\zeta\left(s\right)}{\delta\log^{s}\eta}\right)
	\end{equation*}
	with $\zeta$ denoting the Riemann zeta function,
	\begin{equation*}
		k_{1}=\left(\eta^{1/4}+\eta^{-1/4}\right)/\sqrt{2}
		~~\mbox{and}~~
		k_{2}=\left(\sqrt{\eta}+1\right)/2.
	\end{equation*}

	Let us first simplify $S_{\delta/2}\left(m\vee v\right)$ and then choose
	the parameters $m,\eta,s$ to yield decent constants. Writing $Z = s\log\log\left(\frac{\eta (m \vee v)}{m}\right)+\log\frac{2\zeta\left(s\right)}{\delta\log^{s}\eta}$, we have
	\begin{equation*}
		S_{\delta/2}\left(m\vee v\right) \le (k_1 + 2k_2) \sqrt{(m\vee v \vee Z)Z}.
	\end{equation*}
	Note that $\log\log\left(\frac{\eta (m \vee v)}{m}\right) \ge \log \log \eta$ and therefore $Z\ge \log\frac{2\zeta\left(s\right)}{\delta} \ge \log (2\zeta(s))$. Therefore, if $m \le \log (2\zeta(s))$, we have the slightly simplified bound
	\begin{equation*}
		S_{\delta/2}\left(m\vee v\right) \le (k_1 + 2k_2)\sqrt{(v+Z)Z}.
	\end{equation*}
	Moreover, for $m\ge \eta$ we may upper bound $Z$ by
	\begin{equation*}
		Z \le s \log\left(\frac{\left[2\zeta\left(s\right)\right]^{1/s}}{\delta\log\eta}\log\left(m+v\right)\right)
	\end{equation*}
	Taking $\eta=m=1.8$ and $s=1.05$, one easily confirms that $m\le 3 \le \log (2\zeta(s))$ and, substituting back $k_1$ and $k_2$ as defined above, we have
	\begin{equation*}
		S_{\delta/2}\left(m\vee v\right) \le \sqrt{ 16 \log \prn*{
		\frac{60 \log(m+v)}{\delta}}	v + 16	\log^2 \prn*{
			\frac{60 \log(m+v)}{\delta}	}}
	\end{equation*}

	Finally, noting that $\prn[\big]{X_s -\hat{X}_s}^2 \le 4$ for all $s$, we may substitute $m+v\le 6t$ in the bound above, concluding the proof.
\end{proof}

\subsection{Proof of \Cref{prop:concentration-single-eta}}\label{app:stochastic:signle-eta-proof}

\begin{proof}
	Since $\eta$ is fixed throughout this proof, we drop the explicit dependence on it to simplify notation. Furthermore, we define the normalized/shorthand quantities:
	\begin{equation*}
		\gdiff' \defeq \gdiff / L~,~
		\dbar[t]' \defeq \max\{\dbar[t],\eta\sqrt{\beta}\}~,~
		\Gcoef' \defeq \Gcoef / 64~,~
		\G[t]' \defeq \G[t] / L^2~~\mbox{and}~
		\Ccoef' \defeq \Ccoef/(64L^2).
	\end{equation*}
	With these definitions, the failure probability we wish to bound is
	\begin{equation*}
		\Pr*(\ev^c) = \Pr*(\exists t\le T: \sum_{i < t} \inner{\gdiff'}{x_i - \xopt} < -2\dbar[t]'\sqrt{\Gcoef' \G[t]' + \Ccoef'}).
	\end{equation*}
	Now, define for any $k\ge 0$:
	\begin{equation*}
		s_k \defeq 2^k \dbar[0]' = 2^k \max\{\d[0], \eta \sqrt{\Ccoef} \}
	\end{equation*}
	and in addition define
	\begin{equation*}
		k_t \defeq \ceil*{ \log \frac{\dbar[t]'}{\dbar[0]'} }.
	\end{equation*}
	Note that $k_t$ satisfies the following
	\begin{equation*}
		\dbar[t]\le \dbar[t]' \le s_{k_t} \le 2\dbar[t]'
		~~\mbox{and}~~
		0 \le k_t \le \ceil*{\log\prn*{\frac{t}{4}+1}} \le \log(6t)-1.
	\end{equation*}
	The first set of inequalities follows from the definition of $\dbar[t]'$, $k_t$ and $s_k$, while the latter inequality is due to the fact that $\dbar[t] \le \d[0] + t \eta L\le \d[0] + \frac{t}{4} \eta\sqrt{\beta}$, which follows from the definition of SGD, the triangle inequality, the  assumption $\norm{g_i} \le L$ w.p.~1, and $\beta \ge 16 L^2$.
	
	Writing $\proj(x) = x / \max\{1,\norm{x}\}$ for the projection to the Euclidean unit ball, we now bound the failure probability as follows
	\begin{flalign*}
		\Pr*(\ev^c) & \overeq{(i)} \Pr*(\exists t\le T: \sum_{i < t} \inner{\gdiff'}{
			\proj\prn*{\frac{x_i - \xopt}{s_{k_t}}}} < -\frac{2\dbar[t]'}{s_{k_t}}\sqrt{\Gcoef' \G[t]' + \Ccoef'})
		\\ &
		\overle{(ii)}
		\Pr*(\exists t\le T: \sum_{i < t} \inner{\gdiff'}{
			\proj\prn*{\frac{x_i - \xopt}{s_{k_t}}}} < -\sqrt{\Gcoef' \G[t]' + \Ccoef'})
		\\ &
		\overle{(iii)}
		\sum_{k=0}^{\floor{\log(6T)-1}} 
		\Pr*(\exists t\le T: \sum_{i < t} \inner{\gdiff'}{
			\proj\prn*{\frac{x_i - \xopt}{s_{k}}}} < -\sqrt{\Gcoef' \G[t]' + \Ccoef'}), \label{eq:ollies-awesome-chaining}\numberthis
	\end{flalign*}
	where $(i)$ follows from $\norm{x_i - \xopt} = \d[i] \le \dbar[t] \le s_{k_t}$ (which means that the projection does nothing), $(ii)$ follows from $s_{k_t} \le 2\dbar[t]'$, and $(iii)$ follows from  $0 \le k_t \le \log(6t) \le \log(6 T)$ and a union bound.
	We can now define a nicely behaved stochastic process: for every $i$ and $k$ let
	\begin{equation*}
		\themg \defeq  \inner{\frac{g_i}{L}}{
			\proj\prn*{\frac{x_i - \xopt}{s_{k}}}}
	\end{equation*}
	and note that $\themg$ is adapted to the filtration $\filt_t = \sigma(g_0, g_1, \ldots, g_t)$ (i.e., $\themg \in \filt_{t}$) and satisfies $\abs*{\themg} \le \frac{\norm{g_i}}{L}\le 1$ by Cauchy-Schwarz and $\norm{g_i}\le L$.
	Applying~\Cref{cor:kickass-mg-concentration} with $\hat{X} = 0$ as the predictable sequence, we obtain, for any $k$ and $\delta' \in (0,1)$,
	\begin{equation}
		\Pr*( \exists t \le T: 
		\abs*{\sum_{i < t} \prn*{\themg - \Ex*{ \themg | \filt_{i-1} }}}
		\ge 4 \sqrt{A_{t}(\delta') \sum_{i< t} (\themg)^2 + A_t^2(\delta')}
		) \le \delta',\label{eq:kickass-concentration-conclusion}
	\end{equation}
	where $A_t(\delta') = \log \prn*{
		\frac{60 \log(6t)}{\delta'}}$. Note that 
	\begin{equation*}
		\themg 
	- \Ex*{ \themg | \filt_{i-1} } = \inner{\gdiff'}{
		\proj\prn*{\frac{x_i - \xopt}{s_{k}}}}
	\end{equation*}
	 and that 
	 \begin{equation*}
	 	\sum_{i< t} (\themg)^2 \le \sum_{i< t} \norm{g_i}^2 / L^2 = \G[t]'.
	 \end{equation*} 
 Furthermore, note that, for $\delta' = \frac{\delta}{\log(6 T)}$ we have that $A_t(\delta') \le A_T(\delta') = C = \Gcoef / 32^2 \le \Gcoef' / 16$ and that $A_t^2(\delta') \le A_t^2(\delta') = C^2 \le \Ccoef' / 16$. Substituting to inequality~\eqref{eq:kickass-concentration-conclusion} we conclude that
	\begin{equation*}
		\Pr*(\exists t\le T: \sum_{i < t} \inner{\gdiff'}{
			\proj\prn*{\frac{x_i - \xopt}{s_{k}}}} < -\sqrt{\Gcoef' \G[t]' + \Ccoef'}) \le \frac{\delta}{\log(6 T)},
	\end{equation*}
	for all $k$, and therefore $\Pr*(\ev^c) \le \delta$ by the bound~\eqref{eq:ollies-awesome-chaining}.
\end{proof}

\subsection{Proof of \Cref{lem:concentration-master}}\label{app:stochastic-concentration-master-proof}

\begin{proof}
	Fixing some $k\in\{2,4,8,\ldots\}$ and noting that $C_k =\log \prn*{\frac{60 \log^2(6B)}{2^{-2k}\delta}}$, we may apply~\Cref{prop:concentration-single-eta} with $T=B$, $\Gcoef=\GcoefK$, $\Ccoef=\CcoefK$ and failure probability $2^{-2k}\delta$, giving $1-\Pr*(\ev[B, \GcoefK, \CcoefK]( \eta )) \le 2^{-2k} \delta$ for any $\eta$. Therefore, by the union bound
	\begin{equation*}
		1 - \Pr*(  \bigcap_{j=0, 1, \ldots, 2^k} \ev[B, \GcoefK, \CcoefK]( 2^j \etainit ) ) \le (2^k + 1)2^{-2k} \delta \le 2^{-(k-1)} \delta.
	\end{equation*}
	Applying the union bound once more, we have
	\begin{equation*}
		1 - \Pr*(  \bigcap_{k=2, 4, 8, \ldots} \bigcap_{j=0, 1, \ldots, 2^k} \ev[B, \GcoefK, \CcoefK]( 2^j \etainit ) ) \le \sum_{k=2,4,8,\ldots} 2^{-(k-1)}\delta \le \sum_{k\ge 1} 2^{-k}\delta = \delta.
	\end{equation*}
\end{proof}

\subsection{Proof of \Cref{thm:stochastic-main}}\label{app:stochastic:main-proof}

\begin{proof}
	The bound $B$ on the algorithm's query number is deterministic and therefore follows exactly as in the proof of~\Cref{thm:deterministic-main}. 
	For the remainder of the analysis we assume the event \linebreak $\bigcap_{k=2, 4, 8, \ldots} \bigcap_{j=0, 1, \ldots, 2^k} \ev[B, \GcoefK, \CcoefK]( 2^j \etainit )$ holds, which by \Cref{lem:concentration-master} happens with probability at least $1-\delta$; we will show that, conditional on this event holding, the conclusions of the theorem hold deterministically. (Note that $\ev[B, \GcoefK, \CcoefK]$ implies $\ev[T_k, \GcoefK, \CcoefK]$ for all $T_k \le B$).
	
	Next, we establish the lower bound \eqref{eq:stochastic-T-lb} on $T$. Note that \Cref{alg:main} indeed always returns a point of the form $\bar{x} = \frac{1}{T}\sum_{i< T}x_i(\eta)$ with some $T\ge 1$; the edge case of returning in \cref{line:main-alg-edge-case-termination} corresponds to $T=1$. Moreover, in the typical case of returning in \cref{line:main-alg-normal-termination}, we have $T=T_{\kout}=\floor*{\frac{B}{2\kout}} \ge \frac{B}{4\kout}$ for some $\kout\le B/4$. By \Cref{lem:stochastic-eta-max}, we have 
	\begin{equation}\label{eq:stochastic-kout-bound}
		\kout \le 2 \log \log_+ \frac{\etamax(\GcoefK[0],\CcoefK[0])}{\etainit} \overle{(\star)} 2 \log \log_+ \prn*{\frac{4\cdot 32}{32^2 -2}\cdot \frac{\d[0]}{\etainit L}} \le 2\log \log_+ \frac{\d[0]}{\etainit  L},
	\end{equation}
	where $(\star)$ follows the facts that $\etamax(\Gcoef,\Ccoef)$ is decreasing in $\Gcoef, \Ccoef$, and $\GcoefK[0] \ge 32^2$ and $\CcoefK[0]\ge 32^2 L^2$ by the setting~\eqref{eq:coef-k-setting}. Thus we obtain 
	the claimed lower bound~\eqref{eq:stochastic-T-lb} on $T=T_{\kout}$. 
	
	It remains to show that one of the  conclusions~\eqref{eq:stochastic-main-bound-normal-case} and~\eqref{eq:stochastic-main-bound-edge-case} must hold. When \Cref{alg:main} returns in \cref{line:main-alg-normal-termination}, we apply \Cref{prop:stochastic:subopt-bound} on the final bisection performed by the algorithm, with parameters $T=T_{\kout}$, $\Gcoef=\GcoefK[\kout]$ and $\Ccoef=\CcoefK[\kout]$. In the typical case that $\etainit \le \bt(\etainit)$, the proposition gives
	\begin{equation*}
		\norm{\bar{x}-\xopt} \le \d[0] + \norm{\bar{x}-x_0} \le \frac{5\Gcoef-2}{\Gcoef-2} \d[0] \le \frac{5\cdot 32^2-2}{32^2-2} \d[0] \le 6\d[0]
	\end{equation*}
	(where we have used $\Gcoef \ge \GcoefK[0] \ge 32^2$), and, for some $\eta'\in[\eta,2\eta]$,
	\begin{equation*}
		f(\bar{x}) - f(\xopt) \le \frac{9\alpha - 2}{2(\alpha - 2)}  \cdot \frac{ \d[0] \sqrt{ \Gcoef \G(\eta') + \Ccoef}}{T} 
		= O\prn*{\frac{ \d[0] \sqrt{ \Gcoef \G(\eta') + \Ccoef}}{T} }.
	\end{equation*}
	Letting $C=\log \frac{\log_+ B}{\delta} +  \log\log_+  \frac{\norm{\xopt-x_0}}{\etainit L}$, we note that  $\Gcoef = O(\kout + \log\frac{\log_+ B}{\delta}) = O(C)$ and similarly $\Ccoef = O(C^2 L^2)$ due to the setting~\eqref{eq:coef-k-setting} and upper bound~\eqref{eq:stochastic-kout-bound} on $\kout$. Therefore, $f(\bar{x}) - f(\xopt) \le O\prn[\Big]{ \frac{ \d[0] \sqrt{C\G(\eta')+C^2L^2}}{T}}$ and conclusion~\eqref{eq:stochastic-main-bound-normal-case} holds. 
	
	In the edge case that $\etainit > \bt(\etainit)$ in the final bisection, we separately consider the cases $\etainit \sqrt{\Gcoef \G(\etainit) + \Ccoef} \le 5\d[0]$ and $\etainit \sqrt{\Gcoef \G(\etainit) + \Ccoef} > 5\d[0]$. In the former case, \Cref{prop:stochastic:subopt-bound} gives
	\begin{equation*}
		\norm{\bar{x}-\xopt} \le \d[0] + \norm{\bar{x}-x_0} \le
		\d[0]+\etainit \sqrt{\Gcoef \G(\etainit) + \Ccoef}\le 6\d[0]
	\end{equation*}
	and
	\begin{equation*}
		f(\bar{x}) - f(\xopt) = O\prn*{\frac{ (\d[0] +\etainit \sqrt{\Gcoef \G(\etainit) + \Ccoef})  \sqrt{\Gcoef \G(\etainit) + \Ccoef}}{T}}
		= O\prn*{\frac{ \d[0] \sqrt{\Gcoef \G(\etainit) + \Ccoef}}{T}},
	\end{equation*}
	so conclusion~\eqref{eq:stochastic-main-bound-normal-case} holds as before. In the second case, where  $5\d[0] < \etainit \sqrt{\Gcoef \G(\etainit) + \Ccoef}$, we have  
	\begin{equation*}
		\norm{\bar{x}-\xopt} \le \d[0] + \norm{\bar{x}-x_0} \le
		\d[0]+\etainit \sqrt{\Gcoef \G(\etainit) + \Ccoef} = O\prn*{\etainit \sqrt{\Gcoef \G(\etainit) + \Ccoef}}
	\end{equation*}
	and
	\begin{equation*}
		f(\bar{x}) - f(\xopt) = O\prn*{\frac{ (\d[0] +\etainit \sqrt{\Gcoef \G(\etainit) + \Ccoef})  \sqrt{\Gcoef \G(\etainit) + \Ccoef}}{T}}
		= O\prn*{\frac{ \etainit \prn{\Gcoef \G(\etainit) + \Ccoef}}{T}}.
	\end{equation*}
	Recalling that $\Gcoef=O(C)$ and $\Ccoef=O(C^2L^2)$, we see that conclusion~\eqref{eq:stochastic-main-bound-edge-case} holds.
	
	Finally, if \Cref{alg:main} returns in \cref{line:main-alg-edge-case-termination} instead of \cref{line:main-alg-normal-termination}, we have $\bar{x}=x_0$ and $T=1$, and so conclusion~\eqref{eq:stochastic-main-bound-normal-case} holds trivially, since  $\norm{x_0-\xopt}\le 6\norm{x_0-\xopt}$ and  $f(x_0)-f(\xopt) \le \inner{\grad f(x_0)}{x_0-\xopt} \le \norm{x_0-\xopt}L = O(\d[0]\sqrt{C^2 L^2})$ due to convexity of $f$ and \Cref{ass:lipschitz}.
\end{proof}

\subsection{A corollary for uniform gradient bounds}\label{app:stochastic:main-coro-proof}

The following corollary translates~\Cref{thm:stochastic-main} to the setting where we replace all observed gradient norms by $L$. In it $\lambda$ represents a double-logarithmic factor and we use $\iota < 1$ to indicate low order terms which can be ignored as soon as $B=\Omega(\lambda^2)$.

\begin{corollary}\label{coro:main-stochastic-bound}
	Under \Cref{ass:lipschitz}, for any $\delta\in (0,1)$ and $\veps>0$,
	\Cref{alg:main} with parameters $\GcoefK, \CcoefK$ given by
	\Cref{eq:coef-k-setting}, $B\in \N$ and $\etainit = \frac{\veps}{L^2}$ 
	makes at most $B$ gradient queries and returns $\bar{x}$ such that, with probability at least $1-\delta$,
	\begin{equation*}
		\norm{\bar{x}-\xopt} \le O\prn*{\max\crl*{\norm{x_0-\xopt}, \frac{\sqrt{B(\lambda + \lambda^2 \iota^2)}}{L} \veps}}
	\end{equation*}
	and
	\begin{equation*}
		f(\bar{x}) - f(\xopt) \le O\prn*{
			\frac{\norm{x_0-\xopt} L (\lambda + \lambda^{3/2} \iota)}{\sqrt{B}} 
			+ (\lambda + \lambda^2 \iota^2) \veps 
		}.
	\end{equation*}
	where
	\begin{equation*}
		\lambda \defeq	\log \frac{1}{\delta}\log_+ \frac{B \norm{x_0-\xopt}L }{\veps}~~\mbox{and}~~\iota = \sqrt{\frac{\lambda}{B+\lambda}}.
	\end{equation*}
\end{corollary}

\begin{proof}
	The corollary follows by substitution of $\etainit = \frac{\veps}{L^2}$ 
	into \Cref{thm:stochastic-main}. In particular, the bound \eqref{eq:stochastic-T-lb} becomes
	\begin{equation*}
		T = \Omega\prn*{\frac{B}{\lambda}+1},
	\end{equation*}
	the quantity $C$ in \Cref{thm:stochastic-main} satisfies
	\begin{equation*}
		C = \log \frac{1}{\delta} +  
		\log\log_+  \frac{B\d[0]}{\etainit L} 
		= \lambda,
	\end{equation*}
	and the upper bound on $\norm{\bar{x}-\xopt}$ in~\eqref{eq:stochastic-main-bound-edge-case} is 
	\begin{equation*}
		O\prn*{ \etainit \sqrt{C\G(\etainit)+C^2L^2}}
		=
		O\prn*{ \frac{\veps}{L} \sqrt{\lambda T + \lambda^2 } } = O\prn*{ \frac{\sqrt{T}\sqrt{\lambda + \lambda^2/(1+B/\lambda)}}{L} \veps},
	\end{equation*}
	which, when substituting $T\le B$ and $\iota^2 = \lambda(\lambda+B)$ and  combining with the bound on $\norm{\bar{x}-\xopt}$ in~\eqref{eq:stochastic-main-bound-normal-case} yields the claimed distance bound.
	
	Finally, recalling that $T\ge1$ and $T=\Omega(B/m)$, the suboptimality bound in~\eqref{eq:stochastic-main-bound-normal-case} reads
	\begin{equation*}
		f(\bar{x}) - f(\xopt) \le O\prn*{ \frac{ \d[0]\sqrt{C\G(\eta')+C^2L^2}}{T}} = 
		O\prn*{ \frac{\d[0]L}{\sqrt{B}}\sqrt{\lambda^2 + \frac{\lambda^3}{T}}} 
		= O\prn*{ \frac{\d[0]L (\lambda + \lambda^{3/2}\iota)}{\sqrt{B}} }
	\end{equation*}
	and the suboptimality bound in~\eqref{eq:stochastic-main-bound-edge-case} reads
	\begin{equation*}
		f(\bar{x}) - f(\xopt) \le O\prn*{\frac{\etainit \prn*{C\G(\etainit)+C^2L^2}}{T}}
		=
		O\prn*{ \veps \prn*{\lambda + \frac{\lambda^2}{T}} } = O((\lambda + \lambda^2\iota^2)\veps);
	\end{equation*}
	combined, these yield the claimed suboptimality bound.
\end{proof}

\section{Proofs for \Cref{sec:adaptivity}}

\subsection{Proof of \Cref{thm:adaptivity-to-smoothness}}\label{app:proof-of:thm:adaptivity-to-smoothness}

\begin{proof}
	Recall the following basic property of any $S$-smooth functions \cite[Lemma 3.4]{bubeck2014convex},
	\begin{flalign}\label{eq:basic-smoothness-fact}
		f(u) - f(v) - \inner{\grad f(v)}{u - v} \le \frac{S}{2} \| u - v \|^2
		~~\mbox{for all $u,v\in\R^d$. }
	\end{flalign}
	Using this fact we establish two useful inequalities for our proof. First, for any $\eta\le \frac{1}{S}$ substituting $u=x_{i+1}(\eta)$ and $v=x_i(\eta)$ into~\eqref{eq:basic-smoothness-fact} gives $\frac{\eta}{2} \| g_i(\eta) \|^2 \le f(x_i(\eta)) - f(x_{i+1}(\eta))$. Summing over $i<T$ we obtain
	\begin{flalign}\label{eq:eta-small-G-T-bound} 
		\eta \le \frac{1}{S}
		 \implies \frac{\eta}{2} \G[T](\eta) 
		 \le f(x_0) - f(x_T(\eta))
	 	\le f(x_0) - f(x_\star) \overle{(\star)} \frac{1}{2}S\d[0]^2,
	\end{flalign}
	where $(\star)$ follows from~\eqref{eq:basic-smoothness-fact} with $u=x_0$ and $v=\xopt$. 
	Second, for any $\eta\ge 0$, substituting $x_i(\eta) - \frac{1}{S} g_i(\eta)$ and $v = x_t(\eta)$ into~\eqref{eq:basic-smoothness-fact} yields
	$\norm{g_i(\eta)}^2 \le 2S \brk*{ f(x_i(\eta)) - f\prn*{ x_i(\eta) - \tfrac{1}{S} g_i(\eta)}} \le 2S \brk*{ f(x_i(\eta)) - f\prn*{\xopt}}$. Summing for $i<T$ yields
	\begin{flalign} \label{eq:G-T-bound-by-sum-of-gaps}
		G_T(\eta) \le 2 S \sum_{i < T}\brk*{ f(x_i(\eta)) - f(x_\star)}
		~~\mbox{for all}~~\eta\ge 0.
	\end{flalign}
	
	We now split the proof into three cases based on the value of $\etaout$. In each case, we show that $f(\bar{x}) - f(\xopt)  =  O\left( \frac{S \d[0]^2}{T} \right)$
	which, when combined with the lower bound~\eqref{eq:determinisitic-T-lb} on $T$, gives
	the result.
	
	First, consider the case that $\etaout \ge \frac{1}{2S}$ (and hence also $\etaout \ne \etainit$, so that final the call to $\rootfinding$ returns at \cref{line:bisection-normal-termination}). Then, by \Cref{lem:gm-error-bound-r},
	\begin{flalign*} 
		\frac{1}{T} \sum_{i<T} f(x_i(\etaout)) - f(\xopt) 
		& \le \frac{\rbar(\etaout)\d[0]}{\etaout T} + \frac{\etaout \G(\etaout)}{ 2T}
		\\ &\overle{(i)} \frac{2S \rbar[T](\etalofinal) \d[0]}{T} + \frac{\rbar(\etalofinal) \sqrt{\G[T](\etaout)}}{2 T \sqrt{\Gcoef}} 
		\overeq{(ii)} 
		O\left( \d[0] \frac{S \d[0] + \sqrt{\G[T](\etaout)}}{T} \right) \\
		&\overeq{(iii)} O\left(  \d[0] \frac{ S \d[0] + \sqrt{S \sum_{i<T} \brk*{f(x_i(\etaout)) - f(x_\star)} } }{T} \right).
	\end{flalign*} 
	Above, transition $(i)$ follows from substituting $\etaout \ge \frac{1}{2 S}$, $\etaout \le \frac{\rbar(\etalofinal)}{\sqrt{\Gcoef \G(\etaout) }}$ and $\rbar(\etaout) \le \rbar(\etalofinal)$ (the latter two bounds due to \Cref{lem:output-properties}). 
	Transition $(ii)$ is from \Cref{lem:distance-bound}, where we recall that $\etalofinal \le \bt(\etalofinal)$ by the bisection invariant. Transition $(iii)$ follows from substituting the bound~\eqref{eq:G-T-bound-by-sum-of-gaps}.
	If 
	$\sqrt{ \sum_{i=1}^T f(x_i(\eta)) - f(x_\star) } \le \sqrt{S} \d[0]$ then the desired
	bound holds by squaring this inequality and recalling that $f(\bar{x})\le \frac{1}{T}\sum_{i<T} f(x_i(\etaout))$ by convexity. Conversely, if $S \d[0] \le \sqrt{S \sum_{i=1}^T f(x_i(\eta)) - f(x_\star) }$ then, due to the previous display,
	$$
	\frac{1}{T} \sum_{i=1}^T f(x_i(\eta)) - f(x_\star) = O\left(  \d[0] \frac{\sqrt{S \sum_{i=1}^T f(x_i(\eta)) - f(x_\star) } }{T} \right)
	$$
	and rearranging gives $\sqrt{\sum_{i=1}^T f(x_i(\eta)) - f(x_\star) } = O(\sqrt{S} \d[0])$ again. 
	
	Second, consider the case that $\etainit < \etaout \le \frac{1}{2 S}$ (so that the final call to $\rootfinding$ still returns at \cref{line:bisection-normal-termination}). Let $\etalofinal$ and $\etahifinal$ denote the final values of $\etalo$ and $\etahi$, respectively, so that $\etahifinal \le 2\etaout = \frac{1}{S}$. Let us bound the suboptimality of the $\etahifinal$ iterates: beginning with \Cref{lem:gm-error-bound-r}, we have
	\begin{flalign*}
		\frac{1}{T}\sum_{i< T} f(x_i(\etahifinal)) - f(\xopt) & \le 
		\frac{\rbar(\etahifinal)\d[0]}{\etahifinal T} + \frac{\etahifinal \G(\etahifinal)}{ 2T}
		\overeq{(i)}
		O\prn*{ \frac{\d[0]\sqrt{\G(\etahifinal)} + S\d[0]^2 }{T}
		}
		\\ &
		\overeq{(ii)}
		O\left(  \d[0] \frac{ S \d[0] + \sqrt{S \sum_{i<T} \brk*{f(x_i(\etahifinal)) - f(x_\star)} } }{T} \right)
	\end{flalign*}
	where $(i)$ follows from the bisection invariant $\etahifinal \ge \bt(\etahifinal) = \frac{\rbar(\etahifinal)}{\sqrt{\Gcoef \G(\etahifinal) }}$ and the bound~\eqref{eq:eta-small-G-T-bound}, while $(ii)$ follows from the bound~\eqref{eq:G-T-bound-by-sum-of-gaps}. Therefore, by that same considerations as in the $\etaout \ge \frac{1}{2S}$ case, we conclude that 
	\begin{equation}\label{eq:smooth-etahi-small-bound}
		\sum_{i< T} \brk*{f(x_i(\etahifinal)) - f(\xopt)} = O( S\d[0]^2 ).
	\end{equation}
	If $\etaout = \etahifinal$ then we are done since $f(\bar{x}) \le 	\frac{1}{T}\sum_{i< T} f(x_i(\etahifinal))$ by convexity. Otherwise, applying \Cref{lem:gm-error-bound-r} a final time gives
	\begin{flalign*}
		\sum_{i< T} f(\bar{x}) - f(\xopt)  & \le 
		\frac{\rbar(\etaout)\d[0]}{\etaout T} + \frac{\etaout \G(\etaout)}{2T}
		\overeq{(i)}
		O\prn*{ 
			\frac{\d[0]\sqrt{\G(\etahifinal)} + S\d[0]^2 }{T}
		} 
		\\ &
		\overeq{(ii)}
		O\left(  \d[0] \frac{ S \d[0] + \sqrt{S \sum_{i<T} \brk*{f(x_i(\etahifinal)) - f(x_\star)} } }{T} \right)
		\overeq{(iii)}
		O\prn*{\frac{S\d[0]^2}{T}},
	\end{flalign*}
	where $(i)$ follows form $\etaout \ge \frac{\rbar(\etaout)}{\sqrt{\Gcoef \G(\etahifinal)}}$ (via \Cref{lem:output-properties}) and the bound~\eqref{eq:eta-small-G-T-bound}, $(ii)$ is due to~\eqref{eq:G-T-bound-by-sum-of-gaps}, and for $(iii)$ we substituted our bound~\eqref{eq:smooth-etahi-small-bound}	on the error of the $\etahifinal$ iterates.
	
	Finally, the case where $\etaout =  \etainit < \frac{1}{2S}$ follows identically to the bound \eqref{eq:smooth-etahi-small-bound} since we have $\etaout \ge \bt(\etaout)$ and $\etaout\le \frac{1}{S}$ in that case as well. 
\end{proof}

\subsection{Proof of \Cref{thm:strong-convexity-result}}\label{app:proof-of:thm:strong-convexity-result}

\newcommand{\tileps}{\widetilde{\varepsilon}}
\newcommand{\tilL}{\widetilde{L}}

\begin{proof}
	First, note that by \Cref{thm:stochastic-main}, computing $x\pind{M}$ requires $\sum_{m=1}^M B\pind{m} = 2+4+\cdots +2^M = 2^{M+1} - 2 \le B$ gradients queries as claimed.  Next, note that
	\begin{equation*}
		\sum_{m=1}^M \delta\pind{m} = \delta \sum_{m=1}^M \frac{1}{m(m+1)} = \delta \prn*{1-\frac{1}{M+1}} \le \delta,
	\end{equation*}
	and therefore by the union bound, with probability at least $1-\delta$ the conclusions of \Cref{thm:stochastic-main} hold for all applications of $\pftuner$; we proceed with our analysis conditional on that event.
	
	Note that $\delta\pind{m}  \ge \delta\pind{M} = \Omega(\delta / \log^2 B)$ and $\etainit\pind{m} \ge \etainit\pind{M} = \Omega( \varepsilon / (L^2 B))$. Consequently, we have
	\begin{equation*}
		\lambda \defeq \log \prn*{
			\frac{1}{\delta\pind{M}} \log_+ \frac{B\d[0] }{\etainit\pind{M} L}} = O\prn*{ \log \prn*{\frac{1}{\delta} \log_+ \frac{B\d[0] }{\varepsilon}} }.
	\end{equation*}
	With this, we apply \Cref{thm:stochastic-main} on to bound the suboptimality of $x\pind{m}$ for $m\le M$. Let $T\pind{m}$ be the corresponding $T$ from \Cref{thm:stochastic-main}, and note that $T\pind{m} \ge \max\{1,B\pind{m}/\lambda\}$. Noting also that $\G[T\pind{m}](\eta') / T\pind{m} \le L^2$ for all $\eta'$, we have
	\begin{flalign*}
		f(x\pind{m}) - f(\xopt) &= O\prn*{ \max\crl*{ \frac{\varepsilon(\lambda + \lambda^2/T\pind{m})}{B\pind{m}} , 
		\frac{\norm{x\pind{m-1}-\xopt}L \sqrt{\lambda^2 + \lambda^3 / T\pind{m}}}{\sqrt{B\pind{m}}}}}
		\\ &\le  %
			\max\crl*{ \frac{\tileps}{B\pind{m}} , 
				\frac{\tilL\norm{x\pind{m-1}-\xopt}}{\sqrt{B\pind{m}}}},
	\end{flalign*}
	for some $\tileps = O(\varepsilon \lambda^2)$ and $\tilL = O(L \lambda^{3/2})$, and all $m\le M$. Applying strong convexity, we have that  $\frac{\mu}{2} \| x^{(m-1)} - x_\star \|^2 \le f(x^{(m-1)}) - f(x_\star)$ which implies
	\begin{flalign*}
		f(x\pind{m}) - f(\xopt) & \le \max\crl*{ \frac{\tileps}{B\pind{m}} , 
			\sqrt{\frac{2\tilL^2}{\mu B\pind{m}} ( f(x\pind{m-1}) - f(\xopt) ) }
		} \\&
		\overle{(\star)}
		 \max\crl*{ \frac{\tileps}{B\pind{m}} , \frac{4\tilL^2}{\mu B\pind{m}}, 
		 	\frac{f(x\pind{m-1}) - f(\xopt)}{2}
			 }
	\end{flalign*}
	where $(\star)$ follows from $\sqrt{ab} \le \max\{2a,b/2\}$. Iterating this bound and noting that $2B\pind{m-1} = B\pind{m}=2^m$  we conclude that
	\begin{equation}\label{eq:strong-convex-max-bound}
			f(x\pind{m}) - f(\xopt) \le  \max\crl*{ \frac{\tileps}{B\pind{m}} , \frac{4\tilL^2}{\mu B\pind{m}}, 
				\frac{f(x\pind{0}) - f(\xopt)}{B\pind{m}}
			}.
	\end{equation}
	Finally, the strong convexity and Lipschitz continuity assumptions imply that
	\begin{equation*}
		\frac{\mu}{2} \| x^{(0)} - x_\star \|^2 \le f(x^{(0)}) - f(x_\star) \le L \| x^{(0)} - x_\star \| 
	\end{equation*}
	and therefore $\| x^{(0)} - x_\star \| \le \frac{2L}{\mu}$ and, using Lipschitz continuity again, $f(x^{(0)}) - f(x_\star) \le \frac{2L^2}{\mu} \le \frac{4\tilL^2}{\mu}$. Substituting back into~\eqref{eq:strong-convex-max-bound}, the second and third terms merge.
	Recalling that $B\pind{M} = 2^M = B/2$, and that  $\tileps = O(\varepsilon \lambda^2)$ and $\tilL = O(L \lambda^{3/2})$, we have
	\begin{equation*}
			f(x\pind{M}) - f(\xopt) \le \max\crl*{ \frac{2\tileps}{B} , \frac{8\tilL^2}{\mu B} } = O\prn*{\frac{(L^2 / \mu + \varepsilon)\lambda^3}{B}}
	\end{equation*}
	as required.
\end{proof} 
\section{Additional discussion}\label{sec:discussion}

\subsection{Relaxing the assumption that $\xopt$ is optimal}\label{sec:relaxed-xopt}
Optimality gap bounds obtained via online-to-batch conversion have the appealing property of holding for any comparator $\comparator\in\xset$ and not necessarily a minimizer of $f$~\cite{hazan2016introduction}. Consequently, the parameter-free regret minimization algorithm of \citet{mcmahan2014unconstrained} outputs a point $\bar{x}$ with an error bound of the form
\begin{equation*}
	\E f(\bar{x}) \le %
	f(x') + \varepsilon + O\prn*{\frac{L\norm{x_0 - x'}}{\sqrt{T}}\sqrt{\log\prn*{1+\frac{L\norm{x_0 - x'}}{\varepsilon\sqrt{T}}}}}  
	~\mbox{for all }x'\in\xset.
\end{equation*}
In contrast, we only provide guarantees for $x'=\xopt$, a minimizer of $f$. This can be restrictive in settings where $\norm{x_0-\xopt}$ is very large or possibly infinite, i.e., when the minimum of $f$ is not attained, as is the case in logistic regression on separable data. 

However, the assumption that $\xopt$ is optimal can be relaxed. In particular, our only real requirement from $\xopt$ is that, for every SGD iterate $x_t(\eta)$ evaluated in \Cref{alg:main}, we have $f(x_t(\eta))-f(\xopt) \ge 0$. In the noiseless setting, we may modify \Cref{alg:main} to return the GD iterate with lowest objective value (from all the GD executions combined). Such modified algorithm would satisfy the error bounds in \Cref{thm:deterministic-main} with respect to  an \emph{arbitrary} point $\xopt$: if the algorithm's output has function value smaller than $f(\xopt)$, we are done; otherwise, we have $f(x_t(\eta))-f(\xopt) \ge 0$ for every GD iterate, and our analysis goes through. (Note, however, that we lose the guarantee on the distance between $\xopt$ and the algorithm's output). 
Extension to the stochastic case is more involved since we do not have the privilege of choosing the best SGD iterate; we leave it to future work.

\subsection{An alternative bisection target without gradient norm adaptivity}\label{sec:alternative-bisection}
\Cref{alg:main} is fairly adaptive to stochastic gradient norms, with performance guarantees that depend mainly on observed norms, featuring an a-priori gradient norm bound in low-order terms. Moreover, in the noiseless setting our method requires no a-priori bound on gradient norms and our bounds depend solely on observed norms. 

It is possible, however, to slightly simplify our method and sharpen some of our bounds by forgoing adaptivity to gradient norms. Specifically, if we only seek guarantees that depend on an a-priori gradient norm bound $L$, then it is possible to replace the bisection target $\phi$ defined in~\cref{line:bisection-target-def} of \Cref{alg:main} with
\begin{equation*}
	\phi(\eta) = \frac{\rbar[T](\eta)}{\sqrt{\Gcoef L^2 T}}.
\end{equation*}
Our analysis applies to this modified bisection target as well, but  with $\G(\eta)$ replaced by $L^2T$ throughout. Moreover, this modification allows us to slightly improve two parts of the analysis.

First, we may sharpen the bound on $\etamax$ in \Cref{lem:eta-max,lem:stochastic-eta-max} to $\etamax = O\prn*{\frac{\d[0]}{L\sqrt{T}}  }$, improving our bound on the maximum value of $k$ used in \Cref{alg:main}. 
In the deterministic case, this allows us to establish optimality gap bounds scaling as $\varepsilon + \lambda' \frac{d_0 L}{\sqrt{B}}$ for $\lambda' = O\prn[\Big]{\sqrt{\log\log_+ \frac{d_0 L}{\varepsilon \sqrt{B}}}}$ which satisfies $\lambda'=O(1)$ for $\varepsilon = \frac{d_0 L}{\sqrt{B}}$, similarly to the bounds of previous works, as discussed at the end of \Cref{sec:stochastic}.
	
Second, in the stochastic setting, we may use Blackwell's inequality \cite{blackwell1997large} instead of the time uniform empirical Bernstein. This allows us to take $\GcoefK = 2k +
 O\prn*{\log\prn*{\frac{1}{\delta}\log B}}$, eliminating the additive square logarithmic term stemming from $\CcoefK$ in~\eqref{eq:coef-k-setting}. Consequently, in the stochastic setting we obtain a probability $1-\delta$ optimality gap bound of $\varepsilon + \lambda'' \frac{d_0 L}{\sqrt{B}}$ for $\lambda'' = O\prn[\Big]{{\lambda'\brk[\big]{\lambda' + \sqrt{\log\prn*{\frac{1}{\delta}\log B}} } } }$, with $\lambda'$ defined above. Therefore, in the stochastic case we do not remove $\log\log B$ term entirely, even for the optimal $\varepsilon$. The source of the remaining $\log\log B$ is the union bound we use in the proof of \Cref{prop:concentration-single-eta}, which might be removable via a more careful probabilistic argument.

\end{document}